\documentclass[12pt]{article}

\usepackage[utf8]{inputenc}
\usepackage[english]{babel}
\usepackage{amsmath}
\usepackage{nicefrac}
\usepackage{amsthm}
\usepackage{amsfonts}
\usepackage{amssymb}
\usepackage{verbatim}
\usepackage{color}
\usepackage[arrow, matrix, curve]{xy}
\usepackage{bbm}
\usepackage{eqparbox}
\usepackage[numbers]{natbib}
\usepackage{stmaryrd}
\usepackage{amssymb}
\usepackage{mathrsfs}
\usepackage{pictexwd,dcpic}
\usepackage{mathabx}

\DeclareMathSymbol{\leq}{\mathrel}{symbols}{20}
   
\DeclareMathSymbol{\geq}{\mathrel}{symbols}{21}

\newtheoremstyle{WreschTheoremstyle} 
                        {1.5em}    
                        {2.5em}    
                        {}         
                        {}         
                        {\bfseries}
                        {}        
                        {\newline} 
                        {\raisebox{0.6em}{\thmname{#1}\thmnumber{#2}\thmnote{ (#3)}}}

\newcommand{\R}{\mathbb{R}}
\newcommand{\C}{\mathbb{C}}
\newcommand{\N}{\mathbb{N}}

\newcommand{\K}{\mathcal{K}}

\newcommand{\e}{\varepsilon}
\newcommand{\Lb}{\mathcal{L}}
\renewcommand{\1}{\mathbbm{1}}
\newcommand{\dm}{\mathrm{d}}
\newcommand{\esssup}{\mathrm{ess}\sup}

\newcommand{\bfrac}[2]{\genfrac{}{}{0pt}{}{#1}{#2}}

\newtheorem{Theorem}{Theorem}[section]
\newtheorem{Proposition}[Theorem]{Proposition}
\newtheorem{Corollary}[Theorem]{Corollary}
\newtheorem{Lemma}[Theorem]{Lemma}
\newtheorem{Remark}[Theorem]{Remark}
\newtheorem{Definition}[Theorem]{Definition}

\numberwithin{equation}{section}

\makeatletter
\newcommand{\customlabel}[1]{%
     \stepcounter{ref}%
   \protected@write
\@auxout{}{\string\newlabel{#1}{{\thesatz.\arabic{ref}}{\thepage}{\thesatz.\arabic{ref}}{#1}{}}}%
   \hypertarget{#1}{\thesatz.\arabic{ref}}%
}
\makeatother

\newenvironment{sciabstract}{\begin{quote}}{\end{quote}}

\newcounter{lastnote}

\title{Non-equilibrium dynamics for a Widom-Rowlinson type model with mutations}
\newcommand{\pdftitle} {Non-equilibrium dynamics for a Widom-Rowlinson type model with mutations}
\newcommand{\pdfauthor}{Martin Friesen}

\author{
Martin Friesen\footnote{Department of Mathematics, Bielefeld University, Germany, mfriesen@math.uni-bielefeld.de}
}

\usepackage[plainpages=false,pdfpagelabels=true,bookmarks=true,pdfauthor={\pdfauthor},
pdftitle={\pdftitle}]{hyperref}

\makeatletter
\def\HyPsd@CatcodeWarning#1{}
\makeatother

\begin{document}




\maketitle

\begin{sciabstract}\textbf{Abstract:}
A dynamical version of the Widom-Rowlinsom model in the continuum is considered. The dynamics is modelled by a spatial two-component birth-and-death Glauber process 
where particles, in addition, are allowed to change their type with density dependent rates.
An evolution of states is constructed as the unique weak solution to the associated Fokker-Planck equation. 
Such solution is obtained by means of its correlation functions which belong to a certain Ruelle space.
Existence of a unique invariant measure and ergodicity with exponential rate is established.
The mesoscopic limit is considered, it is related with the verification of the chaos preservation property.
\end{sciabstract}


\section{Introduction}
The study of critical behaviour of complex systems and related invariant states is one of the central problems for statistical and mathematical physics.
Particular classes of complex systems can be modelled either as lattice or as continuous models. 
The Ising model is probably one of the most famous examples on the lattice. Its generalization, known as the Potts model, was introduced in \cite{P52}. 
It has been intensively studied on various lattices, cf., e.g., \cite{W82, CET05, GMRZ06} and the review paper \cite{GHM01}.
A detailed analysis of lattice models with general rates and applications can be found in the well-known book of Liggett \cite{LIGGETT}.
In contrast to lattice models, much less is known for their continuous counterparts, i.e. for continuous interacting particle systems. 

We suppose that all particles are located in $\R^d$, are identical by properties and, indistinguishable. 
The corresponding configuration space (= phase space)
\[
 \Gamma = \{ \gamma \subset \R^d \ | \ |\gamma \cap \Lambda| < \infty \text{ for all compacts } \Lambda \subset \R^d\}
\]
is a Polish space (see \cite{KK06}).
Here, $|A|$ stands for the number of elements in the set $A \subset \R^d$.
A natural analogue of the Potts model is in this case given by a Glauber dynamics with phase space $\Gamma$,
where particles randomly disappear and appear in the space.
The corresponding birth-and-death rates are chosen in such a way that the process has a grand canonical Gibbs measure as reversible measure. 
At this point there exist different choices of such rates. We consider, for simplicity, only the case of a constant death rate, say equal to $1$.
Consequently, the Markov (pre-)generator $L$ takes the heuristic form
\begin{align}\label{GLAUBER}
 (LF)(\gamma) = \sum \limits_{x \in \gamma}(F(\gamma \backslash x) - F(\gamma)) + z\int \limits_{\R^d}e^{-E_{\phi}(x,\gamma)}(F(\gamma \cup x) - F(\gamma))\dm x.
\end{align}
For simplicity of notation, we write $\gamma \backslash x$ and $\gamma \cup x$ instead of $\gamma \backslash \{x\}$ and $\gamma \cup \{x\}$, respectively. 
Here, $z > 0$ is the activity and for a symmetric pair potential $\phi: \R^d \longrightarrow \R$ satisfying the usual conditions such as stability, lower regularity
and, integrability,
\[
 E_{\phi}(x,\gamma^{\pm}) := \begin{cases} \sum \limits_{y \in \gamma^{\pm}}\phi(x-y), & \sum \limits_{y \in \gamma^{\pm}}|\phi(x-y)| < \infty \\ \infty, & \text{ otherwise} \end{cases}
\]
is the relative energy of the configuration $\gamma$ interacting with a particle at position $x \in \R^d$. 
This dynamics has a grand canonical Gibbs measure with activity $z$ and potential $\phi$ as symmetrizing measure.
A general treatment of the construction and uniqueness question for Gibbs measures can be found in \cite{CKP16}.
The associated equilibrium process was studied, by Dirichlet forms, in \cite{KL05}.
Existence of a spectral gap and ergodicity of the equilibrium process was shown in the latter work. The case of bounded volumes was considered in \cite{BCC02}.
As a consequence, we may start the Markov process from any initial state being absolutely continuous w.r.t. the symmetrizing (Gibbs) measure. 
In applications, however, we need to study the time evolution for different classes of initial states. Such initial states may be far from equilibrium.
The construction of non-equilibrium statistical dynamics, in terms of states and their associated correlation functions, can be found in \cite{FKK12G, FKKZ12}. 
Exponential ergodicity for the non-equilibrium dynamics was proved in \cite{KKM10}. Its proof is based on a detailed analysis of
solutions to the Kirkwood-Salzburg equations. Such equations are well-known for the characterization and study of Gibbs measures, cf. \cite{GRE71, MOR77, ZAG82}.
The same results have been obtained by stochastic differential equations in \cite{GKT06}.
The mesoscopic limit, however, was only derived in terms of statistical quantities such as states and correlation functions, cf. \cite{FKK11}. 
A stationary density for the associated kinetic equation solves in such a case the well-known Kirkwood-Monroe equation.

Two-component dynamics are modelled on the configuration space
\[
 \Gamma^2 = \{ \gamma = (\gamma^+, \gamma^-) \in \Gamma \times \Gamma \ | \ \gamma^+ \cap \gamma^- = \emptyset \}.
\]
It is a Polish space, when equipped with the restriction of the product topology.
Let $L_0$ be the Markov operator for two interacting Glauber dynamics. It is for every polynomially bounded cylinder functions $F$ given by
\begin{align*}
 (L_0F)(\gamma) &= \sum \limits_{x \in \gamma^+}(F(\gamma^+ \backslash x, \gamma^-) - F(\gamma))
                 +\sum \limits_{x \in \gamma^-}(F(\gamma^+, \gamma^- \backslash x) - F(\gamma))
 \\ &\ \ \ + z^+\int \limits_{\R^d}e^{-E_{\phi^-}(x, \gamma^-)} e^{-E_{\psi^+}(x,\gamma^+)}(F(\gamma^+ \cup x, \gamma^-) - F(\gamma))\dm x
 \\ &\ \ \ + z^-\int \limits_{\R^d}e^{-E_{\phi^+}(x, \gamma^+)} e^{-E_{\psi^-}(x,\gamma^-)}(F(\gamma^+, \gamma^- \cup x) - F(\gamma))\dm x.
\end{align*}
Here, $z^{\pm} > 0$ are the activities of $\pm$ particles and $\phi^{\pm}, \psi^{\pm}$ are symmetric, non-negative and, integrable.
The pair potentials $\phi^{\pm}$ describe the interaction of $+$ particles with particles of same type, whereas $\psi^{\pm}$ take interactions
with particles of different type into account. 

The particular case $\phi^{\pm} = 0$ is known as the dynamical Widom-Rowlinson model. 
It was introduced for a potential with hard-core in \cite{RW70}. An extension is provided by Potts-type systems in \cite{GH96, GMRZ06}. 
Its dynamical version has been recently studied in \cite{FKK15WRMODEL} where a local evolution of correlation functions 
was constructed and its mesoscopic limit was investigated. Moreover, the dynamical phase transition was shown in the mesoscopic limit of this dynamics.
Another model where the birth-and-death events are replaced by density dependent jumps can be found in \cite{BK16}.

We consider an extension to $\phi^{\pm} \neq 0$ where particles are, in addition, allowed to change their type. Such mutations are described by the Markov operator
\begin{align*}
 (VF)(\gamma) &= \sum \limits_{x \in \gamma^+}e^{-E_{\kappa^+}(x,\gamma^+ \backslash x)}e^{-E_{\tau^+}(x,\gamma^-)}(F(\gamma^+ \backslash x, \gamma^- \cup x) - F(\gamma))
 \\ &\ \ \ + \sum \limits_{x \in \gamma^-}e^{-E_{\kappa^-}(x,\gamma^- \backslash x)} e^{-E_{\tau^-}(x,\gamma^+)}(F(\gamma^+ \cup x, \gamma^- \backslash x) - F(\gamma)).
\end{align*}
The pair potentials $\kappa^{\pm} \geq 0$ take the interactions with particles of same type into account, whereas $\tau^{\pm}$ take the interactions
with particles of different type into account. The Markov operator for the birth-and-death dynamics including mutations is, therefore, given by $L = L_0 + V$.

This work is organized as follows. 
We construct a global evolution of states as the unique weak solution to the associated Fokker-Planck equation, see Proposition \ref{POTTSTH:05}. 
Ergodicity with exponential rate is established in the fourth section, see Proposition \ref{PROP:00}.
The mesoscopic limit is studied in the last section, see Theorem \ref{POTTSTH:06}, \ref{POTTSTH:07} and \ref{POTTSTH:08}.
We establish the chaos preservation property and derive the related system of integro-differential equations describing the 
effective densities of the particle system.

\section{Preliminaries}

\subsection{Space of finite configurations}
Let $\Gamma_0 = \{ \eta \subset \R^d \ | \ |\eta| < \infty\}$ be the space of all finite configuration over $\R^d$. It has the natural decomposition 
$\Gamma_0 = \bigsqcup \limits_{n=0}^{\infty}\Gamma_0^{(n)}$, where $\Gamma_0^{(0)} = \{ \emptyset \}$ and $\Gamma_0^{(n)} = \{ \eta \subset \R^d \ | \ |\eta| = n \}$.
Let $\widetilde{(\R^d)^n}$ be the collection of all $(x_1,\dots, x_n) \in (\R^d)^n$ such that $x_j \neq x_k$ whenever $j \neq k$. Then,
\[
 \mathrm{sym}_n: \widetilde{(\R^d)^n} \longrightarrow \Gamma_0^{(n)}, \ \ (x_1, \dots, x_n) \longmapsto \{ x_1, \dots, x_n \}
\]
is a bijection. We endow $\Gamma_0^{(n)}$ with the euclidean topology on $\widetilde{(\R^d)^n}$ and $\Gamma_0$ with the topology of disjoint units.
Then, $\Gamma_0$ is a locally compact Polish space. The Lebesgue-Poisson measure $\lambda$ on $\Gamma_0$ is defined by
\[
 \lambda = \delta_{\emptyset} + \sum \limits_{n=1}^{\infty}\frac{1}{n!}(\dm x)^{(n)},
\]
where $(\dm x)^{(n)} = \dm x^{\otimes n} \circ \mathrm{sym}_n^{-1}$ is the pullback of the Lebesgue measure on $\widetilde{(\R^d)^n}$.
Let $G: \Gamma_0 \times \Gamma_0 \longrightarrow \R$ be measurable, then
\begin{align}\label{IBP}
 \int \limits_{\Gamma_0}\sum \limits_{\xi \subset \eta}G(\xi, \eta \backslash \xi)\dm \lambda(\eta) = \int \limits_{\Gamma_0}\int \limits_{\Gamma_0}G(\xi,\eta)\dm \lambda(\xi)\dm \lambda(\eta),
\end{align}
whenever one side of the equality is finite for $|G|$, cf. \cite{KK02}. The Lebesgue exponential is for a measurable function $f: \R^d \longrightarrow \R$
defined by $e_{\lambda}(f;\eta) = \prod \limits_{x \in \eta}f(x)$ and satisfies, provided $f \in L^1(\R^d)$,
\[
 \int \limits_{\Gamma_0}e_{\lambda}(f;\eta)\dm \lambda(\eta) = \exp\left( \langle f \rangle \right),
\]
where $\langle f \rangle = \int \limits_{\R^d}f(x)\dm x$.

Let $\Gamma_0^2 = \{ \eta = (\eta^+, \eta^-) \in \Gamma_0 \times \Gamma_0 \ | \ \eta^+ \cap \eta^- = \emptyset \}$ be the two-component space of finite configurations.
It is a locally compact Polish space w.r.t. the product topology. For simplicity of notation, we extend all set-operations component-wise.
Namely, $\eta \cup \xi, \eta \backslash \xi, \xi \subset \eta$ stand for $(\eta^+ \cup \xi^+, \eta^- \cup \xi^-)$, $(\eta^+ \backslash \xi^+, \eta^- \backslash \xi^-)$
and, $\xi^+ \subset \eta^+, \xi^- \subset \eta^-$, where $\eta, \xi \in \Gamma_0^2$.
Since $\Gamma_0 \times \Gamma_0 \backslash \Gamma_0^2$ is a set of measure zero for $\lambda \otimes \lambda$, we define the Lebesgue-Poisson measure on $\Gamma_0^2$ by
$\lambda^2 := \lambda \otimes \lambda|_{\Gamma_0^2}$. Since no confusion can arise we keep the notation $\lambda$ for this measure. 
Thus, for any measurable non-negative function $G$
\[
 \int \limits_{\Gamma_0^2}G(\eta)\dm \lambda(\eta) = \int \limits_{\Gamma_0}\int \limits_{\Gamma_0}G(\eta^+, \eta^-)\dm \lambda(\eta^+, \eta^-).
\]
A function $G: \Gamma_0^2 \longrightarrow \R$ is said to have bounded support if there exists
$N \in \N$ and a compact $\Lambda \subset \R^d$ such that
\[
 G(\eta) = 0, \ \ \text{ whenever } \eta \cap \Lambda^c \neq \emptyset \text{ or } |\eta| > N.
\]
Here, we let $|\eta| := |\eta^+| + |\eta^-|$ and $\eta \cap \Lambda^c := (\eta^+ \cap \Lambda^c, \eta^- \cap \Lambda^c)$.
Denote by $B_{bs}(\Gamma_0^2)$ the space of all bounded functions having bounded support.

\subsection{Harmonic analysis on configuration spaces}
The following is a brief summary of \cite{KK02, FKO13}. Given $G \in B_{bs}(\Gamma_0^2)$, the $K$-transform is defined by
\begin{align}\label{POTTS:12}
 (KG)(\gamma) := \sum \limits_{\eta \Subset \gamma}G(\eta),\ \ \gamma \in \Gamma^2,
\end{align}
where $\Subset$ means that the sum runs over all finite subsets of $\gamma$. Let $\mathcal{FP}(\Gamma^2) := K(B_{bs}(\Gamma_0^2))$. For each $F \in \mathcal{FP}(\Gamma^2)$
there exists $A > 0$, $N \in \N$ and a compact $\Lambda \subset \R^d$ such that $F(\gamma) = F(\gamma \cap \Lambda)$ and
\[
 |F(\gamma)| \leq A (1 + |\gamma \cap \Lambda|)^N, \ \ \gamma \in \Gamma^2,
\]
i.e. $F$ is a polynomially bounded cylinder function. Here, $\gamma \cap \Lambda := (\gamma^+ \cap \Lambda, \gamma^- \cap \Lambda)$. 
The map $K: B_{bs}(\Gamma_0^2) \longrightarrow \mathcal{FP}(\Gamma^2)$ is a positivity preserving isomorphism with inverse
\[
 (K^{-1}F)(\eta) = \sum \limits_{\xi \subset \eta}(-1)^{|\eta \backslash \xi|}F(\xi).
\]
Denote by $K_0, K_0^{-1}$ the restrictions of $K, K^{-1}$ to functions on $\Gamma_0^2$.
Let $\mu$ be a probability measure on $\Gamma^2$ with finite local moments, i.e.
\[
 \int \limits_{\Gamma^2}|\gamma^+ \cap \Lambda|^n |\gamma^- \cap \Lambda|^n \dm \mu(\gamma) < \infty
\]
for all $n \in \N$ and all compacts $\Lambda \subset \R^d$. The correlation function $k_{\mu}$, provided it exists, is defined by
\[
 \int \limits_{\Gamma^2}KG(\gamma)\dm \mu(\gamma) = \int \limits_{\Gamma_0^2}G(\eta)k_{\mu}(\eta)\dm \lambda(\eta), \ \ G \in B_{bs}(\Gamma_0^2).
\]
The correlation function is uniquely determined by such property. Note that $K|G|$ is integrable w.r.t. $\mu$. 
The $K$-transform satisfies $\Vert KG \Vert_{L^1(\Gamma^2,\dm \mu)} \leq \Vert G \Vert_{L^1(\Gamma_0^2,k_{\mu}\dm \lambda)}$
for any $G \in B_{bs}(\Gamma_0^2)$. It can be extended to a bounded linear operator $K: L^1(\Gamma^2, \dm \mu) \longrightarrow L^1(\Gamma_0^2,k_{\mu}\dm \lambda)$
in such a way that \eqref{POTTS:12} holds for $\mu$-a.a. $\gamma \in \Gamma^2$.

Let $\rho: \R^d \longrightarrow [1,\infty)$ be a measurable, locally bounded function and take $\alpha = (\alpha^+, \alpha^-) \in \R^2$.
For simplicity of notation, we let 
\[
 e_{\lambda}(\rho;\eta)e^{\alpha|\eta|} := e_{\lambda}(\rho;\eta^+)e_{\lambda}(\rho;\eta^-)e^{\alpha^+|\eta^+|}e^{\alpha^-|\eta^-|}.
\]
Let $\Lb_{\alpha} := L^1(\Gamma_0^2, e^{\alpha|\cdot|}e_{\lambda}(\rho)\dm \lambda)$ with the norm
\[
 \Vert G \Vert_{\Lb_{\alpha}} = \int \limits_{\Gamma_0^2}|G(\eta)|e^{\alpha^+|\eta^+|}e^{\alpha^-|\eta^-|}e_{\lambda}(\rho;\eta^+)e_{\lambda}(\rho;\eta^-)\dm \lambda(\eta).
\]
Denote by $(\Lb_{\alpha})^*$ the dual Banach space to $\Lb_{\alpha}$. We use the duality
\[
 \langle G, k \rangle := \int \limits_{\Gamma_0^2}G(\eta)k(\eta)\dm \lambda(\eta), \ \ G \in \Lb_{\alpha}
\]
to identify $(\Lb_{\alpha})^*$ with the space of all equivalence class of functions $k$ with the norm
\[
 \Vert k \Vert_{\K_{\alpha}} = \esssup \limits_{\eta \in \Gamma_0^2}\ \frac{|k(\eta)|}{e_{\lambda}(\rho;\eta^+)e_{\lambda}(\rho;\eta^-)} e^{-\alpha^+|\eta^+|}e^{-\alpha^-|\eta^-|}.
\]
Let $\K_{\alpha}$ the Banach space of all such equivalence classes of functions $k$. Then, each $k \in \K_{\alpha}$ satisfies the Ruelle bound
\[
 |k(\eta)| \leq \Vert k \Vert_{\K_{\alpha}}e_{\lambda}(\rho;\eta^+)e_{\lambda}(\rho;\eta^-) e^{\alpha^+|\eta^+|}e^{\alpha^-|\eta^-|}, \ \ \eta \in \Gamma_0^2.
\]
A function $G \in \Lb_{\alpha}$ is called positive definite if $KG \geq 0$. Let $B_{bs}^+(\Gamma_0^2)$ be the cone of all positive definite functions in $B_{bs}(\Gamma_0^2)$.
A function $k \in \K_{\alpha}$ is called positive definite if 
\[
 \langle G,k \rangle = \int \limits_{\Gamma_0^2}G(\eta)k(\eta)\dm \lambda(\eta) \geq 0, \ \ G \in B_{bs}^+(\Gamma_0^2).
\]
Note that any positive definite function $k$ is non-negative.
\begin{Theorem}\cite{F11TC}
 Let $k \in \K_{\alpha}$. The following are equivalent.
 \begin{enumerate}
  \item There exists a unique probability measure $\mu$ having finite local moments such that $k$ is its correlation function, i.e. $k_{\mu} = k$.
  \item $k(\emptyset, \emptyset) = 1$ and $k$ is positive definite.
 \end{enumerate}
\end{Theorem}
Let $\mathcal{P}_{\alpha}$ be the collection of all probability measures $\mu$ with finite local moments such that for each $\mu \in \mathcal{P}_{\alpha}$
there exists a correlation function $k_{\mu} \in \K_{\alpha}$. A metric is given by
\[
 d_{\alpha}(\mu, \nu) := \Vert k_{\mu} - k_{\nu} \Vert_{\K_{\alpha}}.
\]
Then, $(\mathcal{P}_{\alpha},d_{\alpha})$ is a complete metric space which is not separable. 
Note that for each $F \in \mathcal{FP}(\Gamma^2)$ there exists a constant $C_{\alpha}(F) > 0$ such that
\[
 \left| \int \limits_{\Gamma^2}F(\gamma)\dm \mu(\gamma) - \int \limits_{\Gamma^2}F(\gamma)\dm \nu(\gamma)\right| \leq d_{\alpha}(\mu, \nu)C_{\alpha}(F).
\]

\section{Evolution of states}
\subsection{Assumptions}
Suppose that $\phi^{\pm}, \psi^{\pm}, \kappa^{\pm}, \tau^{\pm} \geq 0$ are symmetric, there exists $\alpha = (\alpha^+, \alpha^-) \in \R^2$ and a measurable,
locally bounded function $\rho: \R^d \longrightarrow [1,\infty)$ such that the conditions below are satisfied.
\begin{enumerate}
 \item There exists $h: \R^d \longrightarrow (0,\infty)$ with $h \cdot \rho \in L^1(\R^d)$ such that
 \[
  \sup \limits_{y \in \R^d}\ \frac{g(x-y)}{h(y)} =: C_x < \infty, \ \ x \in \R^d
 \]
 for all $g \in \{ \phi^{\pm}, \psi^{\pm}, \kappa^{\pm}, \tau^{\pm}\}$.
 \item Let $C_{g}(x,\alpha^{\pm}) := \exp\left( e^{\alpha^{\pm}}\int \limits_{\R^d}(1 - e^{-g(x-y)})\rho(y)\dm y \right)$, then
  \begin{align*}
   \sup \limits_{x \in \R^d}\ \left( e^{-\alpha^+}C_{\phi^+}(x,\alpha^+)C_{\psi^+}(x,\alpha^-) + C_{\kappa^-}(x,\alpha^-)C_{\tau^-}(x,\alpha^+)\right) &< 2
   \\ \sup \limits_{x \in \R^d}\ \left( e^{-\alpha^-}C_{\phi^-}(x,\alpha^-)C_{\psi^-}(x,\alpha^+) + C_{\kappa^+}(x,\alpha^+)C_{\tau^+}(x,\alpha^-)\right) &< 2
   \\ \sup \limits_{x \in \R^d}\  C_{\kappa^+}(x,\alpha^+)C_{\tau^+}(x,\alpha^-)e^{\alpha^- - \alpha^+} &< 1
   \\ \sup \limits_{x \in \R^d}\ C_{\kappa^-}(x,\alpha^-)C_{\tau^-}(x,\alpha^+)e^{\alpha^+ - \alpha^-} &< 1.
  \end{align*}
\end{enumerate}
The first condition shows that 
\[
 g(x) \leq C_0 h(x) \leq C_0 h(x)\rho(x), \ \ x \in \R^d
\]
and hence $g$ is integrable. Moreover, we have
\[
 (1 - e^{-g(x-y)})\rho(y) \leq g(x-y)\rho(y) \leq C_x h(y)\rho(y), \ \ y \in \R^d
\]
and hence $C_g(x,\alpha^{\pm})$ is well-defined for all $x \in \R^d$.
\begin{Remark}
 The first condition can be replaced by $g \in C_c(\R^d)$, i.e. $g$ is continuous with compact support.
\end{Remark}
In contrast to classical probability theory, we consider, instead of individual trajectories, 
the associated statistical description. Namely, given an initial state $\mu_0 \in \mathcal{P}_{\alpha}$, 
we seek for a family of states $(\mu_t)_{t \geq 0}$ satisfying the Fokker-Planck equation
\begin{align}\label{FPE}
 \frac{\dm}{\dm t}\int \limits_{\Gamma^2}F(\gamma)\dm \mu_t(\gamma) = \int \limits_{\Gamma^2}LF(\gamma)\dm \mu_t(\gamma), \ \ \mu_t|_{t=0} = \mu_0.
\end{align}

\subsection{Evolution of quasi-observables and correlation functions}
Introduce the cumulative death intensity
\[
 M(\eta) = |\eta^+| + |\eta^-| + \sum \limits_{x \in \eta^+}e^{-E_{\kappa^+}(x, \eta^+ \backslash x)}e^{-E_{\tau^+}(x,\eta^-)} + \sum \limits_{x \in \eta^-}e^{-E_{\kappa^-}(x,\eta^- \backslash x)}e^{-E_{\tau^-}(x,\eta^+)}
\]
and let $D(\widehat{L}) = \{ G \in \Lb_{\alpha} \ | \ M \cdot G \in \Lb_{\alpha}\}$. The operator $\widehat{L} := K_0^{-1}LK_0$
is well-defined on $D(\widehat{L})$ and has the form $\widehat{L} = A + B$, where $(AG)(\eta) = -M(\eta)G(\eta)$. The second operator is given by
\begin{align}
 &\ \notag (BG)(\eta) = z^+ \sum \limits_{\xi \subset \eta}\int \limits_{\R^d}e^{-E_{\phi^+}(x,\xi^+)}e^{-E_{\psi^+}(x,\xi^-)}f_x(\phi^+;\eta^+ \backslash \xi^+)f_x(\psi^+;\eta^- \backslash \xi^-)G(\xi^+ \cup x, \xi^-)\dm x
  \\ \notag &\ \ \ + z^- \sum \limits_{\xi \subset \eta}\int \limits_{\R^d}e^{-E_{\phi^-}(x,\xi^-)}e^{-E_{\psi^-}(x,\xi^+)}f_x(\phi^-;\eta^- \backslash \xi^-)f_x(\psi^-;\eta^+ \backslash \xi^+)G(\xi^+, \xi^- \cup x)\dm x
  \\ \notag &\ \ \ + \sum \limits_{\xi \subset \eta}\sum \limits_{x \in \xi^+}e^{-E_{\kappa^+}(x,\xi^+ \backslash x)}e^{-E_{\tau^+}(x,\xi^-)}f_x(\kappa^+;\eta^+ \backslash \xi^+)f_x(\tau^+;\eta^- \backslash \xi^-)G(\xi^+ \backslash x, \xi^- \cup x)
  \\ \notag &\ \ \ + \sum \limits_{\xi \subset \eta}\sum \limits_{x \in \xi^-}e^{-E_{\kappa^-}(x,\xi^- \backslash x)}e^{-E_{\tau^-}(x,\xi^+)}f_x(\kappa^-;\eta^- \backslash \xi^-)f_x(\tau^-;\eta^+ \backslash \xi^+)G(\xi^+ \cup x, \xi^- \backslash x)
  \\ \label{POTTS:03} &\ \ \ - \sum \limits_{\bfrac{\xi \subset \eta}{\xi \neq \eta}}\sum \limits_{x \in \xi^+}e^{-E_{\kappa^+}(x,\xi^+ \backslash x)}e^{-E_{\tau^+}(x,\xi^-)}f_x(\kappa^+;\eta^+ \backslash \xi^+)f_x(\tau^+;\eta^- \backslash \xi^-)G(\xi)
  \\ \label{POTTS:04} &\ \ \ - \sum \limits_{\bfrac{\xi \subset \eta}{\xi \neq \eta}}\sum \limits_{x \in \xi^-}e^{-E_{\kappa^-}(x,\xi^- \backslash x)}e^{-E_{\tau^-}(x,\xi^+)}f_x(\kappa^-;\eta^- \backslash \xi^-)f_x(\tau^-;\eta^+ \backslash \xi^+)G(\xi)
\end{align}
where $f_x(g;\eta) := e_{\lambda}(e^{-g(x-\cdot)}-1;\eta) = \prod \limits_{y \in \eta}\left( e^{-g(x - y)} - 1\right)$.
Let $\1^*(\eta) = \begin{cases}1, & |\eta| = 0\\ 0, & \text{otherwise}\end{cases}$.
\begin{Theorem}\label{POTTSTH:00}
 The operator $(\widehat{L}, D(\widehat{L}))$ is the generator of an analytic semigroup $(\widehat{T}(t))_{t \geq 0}$ of contractions on $\Lb_{\alpha}$.
 Moreover, $\widehat{T}(t)\1^* = \1^*$ holds and $B_{bs}(\Gamma_0^2)$ is a core.
\end{Theorem}
\begin{proof}
 Observe that $(A, D(\widehat{L}))$ is the generator of a positive, analytic semigroup of contractions on $\Lb_{\alpha}$. 
 Let $B'$ be the operator $B$, where $f_x(g;\eta)$ is replaced by $\prod \limits_{y \in \eta}\left| e^{-g(x - y)} - 1\right|$
 and in the last two terms, see \eqref{POTTS:03} and \eqref{POTTS:04}, the $-$ is replaced by $+$.
 Then, $B'$ is a positive linear operator on $D(\widehat{L})$ and satisfies $|BG| \leq B'|G|$, for any $G \in D(\widehat{L})$. 
 For any $0 \leq G \in D(\widehat{L})$, see $\rho \geq 1$, we get
 \begin{align*}
  \int \limits_{\Gamma_0^2}B'G(\eta)e_{\lambda}(\rho;\eta)e^{\alpha|\eta|}\dm \lambda(\eta)
  \leq \int \limits_{\Gamma_0^2} \beta(\alpha;\eta)G(\eta)e_{\lambda}(\rho;\eta)e^{\alpha|\eta|}\dm \lambda(\eta),
 \end{align*}
 where
 \begin{align*}
  \beta(\alpha;\eta) &= e^{-\alpha^+}C_{\phi^+}(\alpha^+)C_{\psi^+}(\alpha^-)|\eta^+| + e^{-\alpha^-}C_{\phi^-}(\alpha^-)C_{\psi^-}(\alpha^+)|\eta^-|
  \\ &\ \ \ + C_{\kappa^+}(\alpha^+)C_{\tau^+}(\alpha^-)e^{\alpha^- - \alpha^+}\sum \limits_{x \in \eta^+}e^{-E_{\kappa^+}(x,\eta^+ \backslash x)}e^{-E_{\tau^+}(x,\eta^-)}
  \\ &\ \ \ + C_{\kappa^-}(\alpha^-)C_{\tau^-}(\alpha^+)e^{\alpha^+ - \alpha^-}\sum \limits_{x \in \eta^-}e^{-E_{\kappa^-}(x,\eta^- \backslash x)}e^{-E_{\tau^-}(x,\eta^+)}
  \\ &\ \ \ + (C_{\kappa^+}(\alpha^+)C_{\tau^+}(\alpha^-) - 1)\sum \limits_{x \in \eta^-}e^{-E_{\kappa^+}(x,\eta^+)}e^{-E_{\tau^+}(x,\eta^- \backslash x)}
  \\ &\ \ \ + (C_{\kappa^-}(\alpha^-)C_{\tau^-}(\alpha^+) - 1)\sum \limits_{x \in \eta^+}e^{-E_{\kappa^-}(x,\eta^-)}e^{-E_{\tau^-}(x,\eta^+ \backslash x)}
  \\ &\leq \left( e^{-\alpha^+}C_{\phi^+}(\alpha^+)C_{\psi^+}(\alpha^-) + C_{\kappa^-}(\alpha^-)C_{\tau^-}(\alpha^+) - 1\right)|\eta^+|
  \\ &\ \ \ + \left( e^{-\alpha^-}C_{\phi^-}(\alpha^-)C_{\psi^-}(\alpha^+) + C_{\kappa^+}(\alpha^+)C_{\tau^+}(\alpha^-) - 1\right)|\eta^-|
  \\ &\ \ \ + C_{\kappa^+}(\alpha^+)C_{\tau^+}(\alpha^-)e^{\alpha^- - \alpha^+} \sum \limits_{x \in \eta^+}e^{-E_{\kappa^+}(x,\eta^+ \backslash x)}e^{-E_{\tau^+}(x,\eta^-)}
  \\ &\ \ \ + C_{\kappa^-}(\alpha^-)C_{\tau^-}(\alpha^+)e^{\alpha^+ - \alpha^-} \sum \limits_{x \in \eta^-}e^{-E_{\kappa^-}(x,\eta^- \backslash x)}e^{-E_{\tau^-}(x,\eta^+)}
  \\ &\leq a(\alpha)M(\eta)
 \end{align*}
 for some $a(\alpha) \in (0,1)$. Consequently, we obtain
 \begin{align}\label{POTTS:09}
  \int \limits_{\Gamma_0^2}B'G(\eta)e^{\alpha|\eta|}e_{\lambda}(\rho;\eta)\dm \lambda(\eta) \leq a(\alpha)\int \limits_{\Gamma_0^2}M(\eta)G(\eta)\dm \lambda(\eta).
 \end{align}
 Take $r \in (0,1)$ such that $\frac{a(\alpha)}{r} < 1$, then
 \[
  \int \limits_{\Gamma_0^2}\left( A + \frac{1}{r}B'\right)G(\eta)e_{\lambda}(\rho;\eta)e^{\alpha|\eta|}\dm \lambda(\eta) \leq 0, \ \ 0 \leq G \in D(\widehat{L}).
 \]
 Hence, $(A + B', D(\widehat{L}))$ is the generator of a positive, strongly continuous semigroup of contractions, cf. \cite[Theorem 2.2]{TV06}.
 By \cite[Theorem 1.1]{AR91}, this semigroup is analytic.
 Moreover, \cite[Theorem 1.2]{AR91} implies that $(A + B, D(\widehat{L}))$ is the generator of an analytic semigroup of contractions on $\Lb_{\alpha}$.
 Since $\1^* \in D(\widehat{L})$ and $\widehat{L}\1^* = 0$, it follows that $\widehat{T}(t)\1^* = \1^*$. 
 In order to see that $B_{bs}(\Gamma_0^2)$ is a core, let $(\Lambda_n)_{n \in \N}$ be an increasing sequence of compacts in $\R^d$ and $G \in D(\widehat{L})$.
 Define 
 \[
  G_n(\eta) := \begin{cases}G(\eta), & |\eta| \leq n \text{ and } \eta \subset \Lambda_n\\ 0, & \text{ otherwise }\end{cases}, 
 \]
 then $G_n \in B_{bs}(\Gamma_0)$, $G_n \longrightarrow G$ a.e. as $n \to \infty$ and $|G_n| \leq |G|$, for all $n \in \N$. 
 Dominated convergence yields $G_n \longrightarrow G$ in $\Lb_{\alpha}$. Moreover, dominated convergence also implies
 $\widehat{L}G_n \longrightarrow \widehat{L}G$ a.e., as $n \to \infty$. Since $|MG_n| \leq M|G|$ and 
 $|BG_n| \leq B'|G_n| \leq B'|G|$ applying again dominated convergence shows that $\widehat{L}G_n \longrightarrow \widehat{L}G$ in $\Lb_{\alpha}$.
 Hence, $(\widehat{L}, D(\widehat{L}))$ is the closure of $(\widehat{L}, B_{bs}(\Gamma_0^2))$.
\end{proof}
Denote by $\widehat{T}(t)^*$ the adjoint semigroup on $\K_{\alpha}$ and by $(\widehat{L}^*, D(\widehat{L}^*))$ its generator. This operator is, by definition,
the adjoint operator to $\widehat{L}$. Let $g_0, g_1: \R^d \longrightarrow \R_+$ be given with $(1- e^{-g_j(x-\cdot)})\rho \in L^1(\R^d)$ for all $x \in \R^d$ and $j= 0,1$.
Let $\mathcal{Q}_x(g_0,g_1)$ be a linear operator on $\K_{\alpha}$ given by
\[
 \mathcal{Q}_x(g_0, g_1)k(\eta) = \int \limits_{\Gamma_0^2}f_x(g_0;\xi^+)f_x(g_1;\xi^-)k(\eta \cup \xi)\dm \lambda(\xi), \ \ x \in \R^d.
\]
This operator satisfies
\begin{align}\label{POTTS:14}
 |\mathcal{Q}_x(g_0,g_1)k(\eta)| \leq e^{\alpha|\eta|}e_{\lambda}(\rho;\eta)C_{g_0}(x, \alpha^+) C_{g_1}(x, \alpha^-)\Vert k \Vert_{\K_{\alpha}}.
\end{align}
Let $L^{\Delta}$ be given by
\begin{align*}
 (L^{\Delta}k)(\eta) = &- |\eta|k(\eta)
 \\ &- \sum \limits_{x \in \eta^-}e^{-E_{\kappa^+}(x,\eta^+)}e^{-E_{\tau^+}(x,\eta^- \backslash x)}\mathcal{Q}_x(\kappa^+, \tau^+)k(\eta)
 \\ &- \sum \limits_{x \in \eta^+}e^{-E_{\kappa^-}(x,\eta^-)}e^{-E_{\tau^-}(x,\eta^+ \backslash x)}\mathcal{Q}_x(\tau^-,\kappa^+)k(\eta)
 \\ &+ z^+ \sum \limits_{x \in \eta^+}e^{-E_{\phi^+}(x,\eta^+ \backslash x)}e^{-E_{\psi^+}(x,\eta^-)}\mathcal{Q}_x(\phi^+,\psi^+)k(\eta^+ \backslash x, \eta^-)
 \\ &+ z^- \sum \limits_{x \in \eta^-}e^{-E_{\phi^-}(x,\eta^- \backslash x)}e^{-E_{\psi^-}(x,\eta^+)}\mathcal{Q}_x(\psi^-, \phi^-)k(\eta^+, \eta^- \backslash x)
 \\ &+ \sum \limits_{x \in \eta^+}e^{-E_{\kappa^+}(x,\eta^+ \backslash x)}e^{-E_{\tau^+}(x,\eta^-)}\mathcal{Q}_x(\kappa^+, \tau^+)k(\eta^+ \backslash x, \eta^-)
 \\ &+ \sum \limits_{x \in \eta^-}e^{-E_{\kappa^-}(x,\eta^- \backslash x)}e^{-E_{\tau^-}(x,\eta^+)}\mathcal{Q}_x(\tau^-,\kappa^-)k(\eta^+, \eta^- \backslash x).
\end{align*}
The next lemma shows that $L^{\Delta}$ can be identified with $\widehat{L}^*$.
\begin{Lemma}
 There exists $M(\alpha^+), N(\alpha^-) > 0$ such that for any $\beta^+ < \alpha^+$ and $\beta^- < \alpha^-$ 
 \[
  \Vert L^{\Delta}k \Vert_{\K_{\alpha}} \leq \left(\frac{M(\alpha^+)}{\alpha^+ - \beta^+} + \frac{N(\alpha^-)}{\alpha^- - \beta^-}\right)\Vert k \Vert_{\K_{\beta}}.
 \]
 Consequently, $L^{\Delta}$ is a bounded linear operator from $\K_{\beta}$ to $\K_{\alpha}$.
 Consider $L^{\Delta}$ on its maximal domain
 \[
  D(L^{\Delta}) = \{ k \in \K_{\alpha}\ | \ L^{\Delta}k \in \K_{\alpha} \},
 \]
 then $L^{\Delta} = \widehat{L}^*$. 
\end{Lemma}
\begin{proof}
 The first assertion follows by \eqref{POTTS:14}, $e^{-E_g(x,\eta^{\pm})} \leq 1$ and,
 \[
  |\eta^{\pm}|e^{-(\alpha^{\pm} - \beta^{\pm})|\eta^{\pm}|} \leq \frac{1}{e(\alpha^{\pm} - \beta^{\pm})}, \ \ \eta \in \Gamma_0^2.
 \]
 For the second assertion, let $G \in D(\widehat{L})$ and $k \in D(\widehat{L}^*)$, then
 \[
  \int \limits_{\Gamma_0^2}G(\eta)(\widehat{L}^*k)(\eta)\dm \lambda(\eta) = \int \limits_{\Gamma_0^2}(\widehat{L}G)(\eta)k(\eta)\dm \lambda(\eta)
  = \int \limits_{\Gamma_0^2}G(\eta)(L^{\Delta}k)(\eta)\dm \lambda(\eta),
 \]
 where the last equality follows by \eqref{IBP} and a simple computation. Thus $L^{\Delta}k = \widehat{L}^*k$ and $D(\widehat{L}^*) \subset D(L^{\Delta})$. 
 Conversely, let $k \in D(L^{\Delta})$.
 Then, for any $G \in D(\widehat{L})$, above equality implies $k \in D(\widehat{L}^*)$ and $\widehat{L}^*k = L^{\Delta}k$.
\end{proof}
Since $\K_{\alpha}$ is not reflexive, $\widehat{T}(t)^*$ is, in general, not strongly continuous. However, it is continuous w.r.t. the weak topology
$\sigma(\K_{\alpha}, \Lb_{\alpha})$. Here, $\sigma(\K_{\alpha}, \Lb_{\alpha})$ is the smallest topology such that all functionals
$G \longmapsto \langle G, k \rangle$ are continuous for any $k \in \K_{\alpha}$. It is well-known, see \cite{ENG00}, that $\widehat{T}(t)^*$ leaves
the proper subspace $\K_{\alpha}^{\odot} := \overline{D(L^{\Delta})}$ invariant. 
Moreover, the restriction $\widehat{T}(t)^{\odot} := \widehat{T}(t)^*|_{\K_{\alpha}^{\odot}}$ is a strongly continuous semigroup with generator 
$\widehat{L}^{\odot}k = L^{\Delta}k$,
\[
 D(\widehat{L}^{\odot}) = \{ k \in D(L^{\Delta}) \ | \ L^{\Delta}k \in \K_{\alpha}^{\odot} \}.
\]
Thus, for any $k_0 \in D(\widehat{L}^{\odot})$, $k_t := \widehat{T}(t)^*k_0$ is the unique classical solution to 
\begin{align}\label{CORR}
 \frac{\partial k_t}{\partial t} = L^{\Delta}k_t, \ \ k_t|_{t=0} = k_0.
\end{align}
in $\K_{\alpha}$. Such system of equations is an analogue of the BBGKY-hierarchy known in the physical literature (see \cite{S88}).
Our aim is to get uniqueness for the the weak formulation
\begin{align}\label{WCORR}
 \frac{\dm}{\dm t}\int \limits_{\Gamma_0^2}G(\eta)k_t(\eta)\dm \lambda(\eta) = \int \limits_{\Gamma_0^2}\widehat{L}G(\eta)k_t(\eta)\dm \lambda(\eta), \ \ k_t|_{t=0} = k_0.
\end{align}
To this end, we use the topology of uniform convergence on compact subsets of $\Lb_{\alpha}$ on $\K_{\alpha}$. 
A basis of neighbourhoods around $0$ is given by sets of the form
\begin{align}\label{POTTS:10}
 \{ k \in \K_{\alpha} \ | \ \sup \limits_{G \in K}\ | \langle G, k \rangle | < \e \},
\end{align}
where $\e > 0$ and $K \subset \Lb_{\alpha}$ is a compact, cf. \cite{WZ06}. Denote by $\mathcal{C}$ the topology generated by 
the basis of neighbourhoods \eqref{POTTS:10}. Note that $\mathcal{C}$ coincides with $\sigma(\K_{\alpha}, \Lb_{\alpha})$
on norm-bounded sets, cf. \cite[Lemma 1.10]{WZ06}.
\begin{Definition}
 Given $k_0 \in \K_{\alpha}$, a weak solution \eqref{WCORR} is a family $(k_t)_{t \geq 0} \subset \K_{\alpha}$ being continuous w.r.t. $\mathcal{C}$ and
 \begin{align}\label{POTTS:00}
  \langle G, k_t \rangle = \langle G, k_0 \rangle + \int \limits_{0}^{t}\langle \widehat{L}G, k_s \rangle \dm s, \ \ G \in B_{bs}(\Gamma_0^2)
 \end{align}
 holds for all $t \geq 0$.
\end{Definition}
\begin{Theorem}\label{POTTSTH:01}
 For any $k \in \K_{\alpha}$ there exists a unique weak solution to \eqref{WCORR}, given by $k_t = \widehat{T}(t)^*k_0$.
 Moreover, the following holds:
 \begin{enumerate}
  \item For any $G \in D(\widehat{L})$, $t \longmapsto \langle G, k_t \rangle$ is continuously differentiable and satisfies \eqref{WCORR} for each $t \geq 0$.
  \item If $k_0 \in \K_{\beta}$ for some $\beta^+ < \alpha^+$ and $\beta^- < \alpha^-$, then $k_t$ is continuous w.r.t. to the norm in $\K_{\alpha}$.
 \end{enumerate}
\end{Theorem}
\begin{proof}
 Since $\widehat{L}$ is the generator of a strongly continuous semigroup, the first assertion follows by \cite[Theorem 2.1]{WZ06}.
 \\ 1. The contraction property implies $\Vert k_t \Vert_{\K_{\alpha}} \leq \Vert k_0 \Vert_{\K_{\alpha}}$ and hence, by $\widehat{L}G \in \Lb_{\alpha}$,
 we see that $s \longmapsto \langle \widehat{L}G, k_s \rangle$ is continuous. By \eqref{POTTS:00}, we see that
 $t \longmapsto \langle G, k_t\rangle$ is continuously differentiable and satisfies \eqref{WCORR} for any $t \geq 0$.
 \\ 2. If $k_0 \in \K_{\beta}$, then $L^{\Delta}k_0 \in \K_{\alpha}$ and hence $k_0 \in D(L^{\Delta}) \subset \K_{\alpha}^{\odot}$.
\end{proof}

\subsection{Positive definiteness}
In this section we show existence and uniqueness of weak solutions to \eqref{FPE}. 
\begin{Lemma}
 There exists a measurable set $\Gamma_{\infty}^2 \subset \Gamma^2$ such that the following holds:
 \begin{enumerate}
  \item For each $\mu \in \mathcal{P}_{\alpha}$ we have $\mu(\Gamma_{\infty}^2) = 1$.
  \item We have $E_{g}(x,\gamma^{\pm}) < \infty$ for all $x \in \R^d$, $(\gamma^+, \gamma^-) \in \Gamma_{\infty}^2$ and, $g \in \{ \phi^{\pm}, \psi^{\pm}, \kappa^{\pm}, \tau^{\pm}\}$.
  \item For each $F \in \mathcal{FP}(\Gamma^2)$, the action $LF(\gamma)$ is well-defined for any $\gamma \in \Gamma_{\infty}^2$.
  Moreover, for any $\mu \in \mathcal{P}_{\alpha}$ we have $F, LF \in L^1(\Gamma^2, \dm \mu)$.
 \end{enumerate}
\end{Lemma}
\begin{proof}
 Set $\Gamma_{\infty}^2 := \left \{ \gamma \in \Gamma \ \bigg| \ \sum \limits_{y \in \gamma^+}h(y) + \sum \limits_{y \in \gamma^-}h(y) < \infty \right \}$. 
 Let $\mu \in \mathcal{P}_{\alpha}$ and $k_{\mu}$ be its correlation function. Then,
 \begin{align*}
  \int \limits_{\Gamma^2} \left(  \sum \limits_{y \in \gamma^+}h(y) + \sum \limits_{y \in \gamma^-}h(y) \right)\dm \mu(\gamma)
   &= \int \limits_{\R^d}h(y)k_{\mu}^{(1,0)}(y)\dm y + \int \limits_{\R^d}h(y)k_{\mu}^{(0,1)}(y)\dm y
  \\ &\leq (e^{\alpha^+} + e^{\alpha^-})\Vert k_{\mu}\Vert_{\K_{\alpha}} \int \limits_{\R^d}h(y)\rho(y)\dm y < \infty.
 \end{align*}
 This shows that $\Gamma_{\infty}^2$ has full support for $\mu$. Let $E_{g}(x,\gamma) := \sum \limits_{y \in \gamma}g(x-y)$, 
 where $g \in \{ \phi^{\pm}, \psi^{\pm}, \kappa^{\pm}, \tau^{\pm}\}$. Then,
 \[
  E_{g}(x,\gamma^{\pm}) \leq C_x\left( \sum \limits_{y \in \gamma^+}h(y) + \sum \limits_{y \in \gamma^-}h(y)\right) < \infty, \ \ \gamma \in \Gamma_{\infty}^2, \ x \in \R^d
 \]
 For the last property let $G \in B_{bs}(\Gamma_0^2)$ and take $F = KG \in \mathcal{FP}(\Gamma^2)$.
 By the second property, $LF(\gamma)$ is well-defined for any $\gamma \in \Gamma_{\infty}^2$. 
 Moreover, we have $G, \widehat{L}G \in \Lb_{\alpha} \subset L^1(\Gamma_0^2, k_{\mu}\dm \lambda)$. 
 Since $K: L^1(\Gamma_0^2, k_{\mu}\dm \lambda) \longrightarrow L^1(\Gamma^2, \dm \mu)$ is continuous, by $LF = K\widehat{L}G$, it follows that
 $F, LF \in L^1(\Gamma^2, \dm \mu)$.
\end{proof}
For $\mu \in \mathcal{P}_{\alpha}$ and $F \in L^1(\Gamma^2, \dm \mu)$ let
\[
 \langle \langle F,\mu \rangle \rangle = \int \limits_{\Gamma^2}F(\gamma)\dm \mu(\gamma).
\]
Below we give the definition of a weak solution to \eqref{FPE}.
\begin{Definition}
 Let $\mu_0 \in \mathcal{P}_{\alpha}$, a weak solution to \eqref{FPE} is a family $(\mu_t)_{t \geq 0} \subset \mathcal{P}_{\alpha}$ such that
 for any $F \in \mathcal{FP}(\Gamma^2)$, $t \longmapsto \langle \langle LF, \mu_t \rangle \rangle$ is locally integrable and satisfies
 \[
  \langle \langle F,\mu_t \rangle \rangle = \langle \langle F, \mu_0 \rangle \rangle + \int \limits_{0}^{t}\langle \langle LF, \mu_s \rangle \rangle \dm s, \ \ t \geq 0.
 \]
\end{Definition}
The next theorem establishes uniqueness for weak solutions to \eqref{FPE}.
\begin{Theorem}
 The Fokker-Planck equation \eqref{FPE} has at most one weak solution $(\mu_t)_{t \geq 0} \subset \mathcal{P}_{\alpha}$ such that its correlation functions 
 $(k_{\mu_t})_{t \geq 0}$ satisfy
 \begin{align}\label{POTTS:01}
  \sup \limits_{t \in [0,T]}\ \Vert k_{\mu_t} \Vert_{\K_{\alpha}} < \infty, \ \ \forall T > 0.
 \end{align}
\end{Theorem}
\begin{proof}
 Let $(\mu_t)_{t \geq 0} \subset \mathcal{P}_{\alpha}$ be a weak solution to \eqref{FPE}, denote by $(k_{\mu_t})_{t \geq 0}$ the associated family of
 correlation functions. Let $G \in B_{bs}(\Gamma_0^2) \subset D(\widehat{L})$ and $F = KG \in \mathcal{FP}(\Gamma^2)$. 
 Then $G, \widehat{L}G \in \Lb_{\alpha} \subset L^1(\Gamma_0^2, k_{\mu_t}\dm \lambda)$, $t \geq 0$ and hence $F, LF$ belong to $L^1(\Gamma^2, \dm \mu_t)$.
 By $\langle \langle F, \mu_t\rangle \rangle = \langle G, k_{\mu_t}\rangle$, $\langle \langle LF, \mu_t\rangle \rangle = \langle \widehat{L}G, k_{\mu_t}\rangle$ and, 
 \eqref{FPE} it follows that  $t \longmapsto \langle \widehat{L}G, k_{\mu_t}\rangle$ is locally integrable and \eqref{POTTS:00} holds. 
 In particular, $k_t$ is continuous w.r.t. $\sigma(\K_{\alpha}, \Lb_{\alpha})$. By \eqref{POTTS:01} it is continuous w.r.t. $\mathcal{C}$, which shows the assertion.
\end{proof}
Below we state our main result for this section.
\begin{Proposition}\label{POTTSTH:05}
 For each $\mu_0 \in \mathcal{P}_{\alpha}$ there exists exactly one weak solution $(\mu_t)_{t \geq 0} \subset \mathcal{P}_{\alpha}$ to \eqref{FPE}
 such that its correlation functions satisfy \eqref{POTTS:01}. This solution is uniquely determined by the associated family of correlation functions
 $k_{\mu_t} = \widehat{T}(t)^*k_{\mu_0}$.
\end{Proposition}
The following is a particular case of \cite{KKM08}.
\begin{Corollary}
 For any $\mu_0 \in \mathcal{P}_{\alpha}$ there exists a Markov function $(X_t^{\mu_0})_{t \geq 0}$ on the configuration space $\Gamma^2$ with the initial
 distribution $\mu_0$ associated with the generator $L$.
\end{Corollary}
Since uniqueness was already shown, it remains to prove existence of a weak solution to \eqref{FPE}. To this end, it suffices to show that
$k_t := \widehat{T}(t)^*k_{\mu_0} \in \K_{\alpha}$ is positive definite for each $t \geq 0$. 

\subsubsection*{Step 1: Evolution of local densities}
Let $R_{\delta}(x) := \frac{e^{-\delta|x|^2}}{1 + \delta \rho(x)}$ and $z_{\delta}^{\pm}(x) := R_{\delta}(x)z^{\pm}$, $\delta > 0$. Then,
\begin{enumerate}
 \item $R_{\delta}(x) \longrightarrow 1$ as $\delta \to 0$ for any $x \in \R^d$.
 \item $R_{\delta}(x) \leq \min\{ 1, e^{-\delta|x|^2}\}$ for any $x \in \R^d$, $\delta > 0$.
 \item $\rho \cdot R_{\delta}$ is integrable for any $\delta > 0$.
\end{enumerate}
Denote by $L_{\delta}$ the associated Markov operator. We consider this operator on all measurable functions $F: \Gamma_0^2 \longrightarrow \R$.
Note that $L_{\delta}F$ is, in general, not bounded on $\Gamma_0^2$ even if $F$ is bounded. Let
\begin{align*}
 D_{\delta}(\eta) &= |\eta| 
 \\ &\ \ \ + \int \limits_{\R^d}z^+_{\delta}(x)e^{-E_{\phi^+}(x,\eta^+)}e^{-E_{\psi^+}(x,\eta^-)}\dm x 
           + \int \limits_{\R^d} z^{-}_{\delta}(x)e^{-E_{\phi^-}(x,\eta^-)}e^{-E_{\psi^-}(x,\eta^+)}\dm x
 \\ &\ \ \ + \sum \limits_{x \in \eta^+}e^{-E_{\kappa^+}(x,\eta^+ \backslash x)}e^{-E_{\tau^+}(x,\eta^-)}
           + \sum \limits_{x \in \eta^-}e^{-E_{\kappa^-}(x,\eta^- \backslash x)} e^{-E_{\tau^-}(x,\eta^+)}
\end{align*}
and $\mathcal{D}_{\delta} = \{ R \in L^1(\Gamma_0^2, \dm \lambda) \ | \ D_{\delta} \cdot R \in L^1(\Gamma_0^2,\dm \lambda) \}$.
Then $(-D_{\delta}, \mathcal{D}_{\delta})$ is the generator of a positive analytic semigroup of contractions on $L^1(\Gamma_0^2,\dm \lambda)$.
Let $\mathcal{Q}_{\delta}$ be another (positive) operator on $\mathcal{D}_{\delta}$ given by
\begin{align*}
 (\mathcal{Q}_{\delta}R)(\eta) &= \int \limits_{\R^d}R(\eta^+ \cup x, \eta^-)\dm x 
           + \sum \limits_{x \in \eta^+}z^+_{\delta}(x)e^{-E_{\phi^+}(x,\eta^+ \backslash x)}e^{-E_{\psi^+}(x,\eta^-)}R(\eta^+ \backslash x, \eta^-)
 \\ &\ \ \ + \int \limits_{\R^d}R(\eta^+, \eta^- \cup x)\dm x
           + \sum \limits_{x \in \eta^-}z^-_{\delta}(x)e^{-E_{\phi^-}(x,\eta^- \backslash x)}e^{-E_{\psi^-}(x,\eta^+)}R(\eta^+, \eta^- \backslash x)
 \\ &\ \ \ + \sum \limits_{x \in \eta^-}e^{-E_{\kappa^+}(x,\eta^+)}e^{-E_{\tau^+}(x,\eta^- \backslash x)}R(\eta^+ \cup x, \eta^- \backslash x)
 \\ &\ \ \ + \sum \limits_{x \in \eta^+}e^{-E_{\kappa^-}(x,\eta^-)}e^{-E_{\tau^-}(x,\eta^+ \backslash x)}R(\eta^+ \backslash x,\eta^- \cup x).
\end{align*}
Then
\[
 \int \limits_{\Gamma_0^2}(\mathcal{Q}_{\delta}R)(\eta)\dm \lambda(\eta) = \int \limits_{\Gamma_0^2}D_{\delta}(\eta)R(\eta)\dm \lambda(\eta), \ \ 0 \leq R \in \mathcal{D}_{\delta}.
\]
Consequently, there exists an extension $(\mathcal{J}_{\delta}, D(\mathcal{J}_{\delta}))$ of $(-D_{\delta} + \mathcal{Q}_{\delta},\mathcal{D}_{\delta})$ such that
$\mathcal{J}_{\delta}$ is the generator of a sub-stochastic semigroup $(S_{\delta}(t))_{t \geq 0}$ on $L^1(\Gamma_0^2,\dm \lambda)$, cf. \cite[Theorem 2.2]{TV06}.
\begin{Lemma}
 $\mathcal{D}_{\delta}$ is a core for $\mathcal{J}_{\delta}$. Moreover, $S_{\delta}(t)$ leaves $L^1(\Gamma_0^2, (1+|\cdot|)\dm \lambda)$ invariant.
\end{Lemma}
\begin{proof}
 Let $V(\eta) = |\eta|$, we want to find a constant $c = c(\delta) > 0$ such that
 \begin{align}\label{POTTS:02}
  L_{\delta}V(\eta) \leq c(\delta)(1 + V(\eta)) - \frac{1}{2}D_{\delta}(\eta), \ \ \eta \in \Gamma_0^2.
 \end{align}
 In such a case the assertion follows from \cite[Proposition 5.1]{TV06}. Observe that
 \[
  (L_{\delta}V)(\eta) \leq - |\eta| + \langle z_{\delta}^+ \rangle + \langle z_{\delta}^- \rangle.
 \]
 Then \eqref{POTTS:02} holds, provided
 \[
  \langle z_{\delta}^+ \rangle + \langle z_{\delta}^- \rangle + \frac{1}{2}D_{\delta}(\eta) \leq (1 + c)|\eta| + c.
 \]
 By $D_{\delta}(\eta) \leq 2 |\eta| + \langle z_{\delta}^+ \rangle + \langle z_{\delta}^- \rangle$ this holds true, provided
 \[
  \frac{3}{2}\left( \langle z_{\delta}^+ \rangle + \langle z_{\delta}^- \rangle \right) + |\eta| \leq (1+c)|\eta| + c.
 \]
 Above inequality is satisfied if $c > 0$ is such that $c > \frac{3}{2}\left( \langle z_{\delta}^+ \rangle + \langle z_{\delta}^- \rangle \right)$.
\end{proof}
Let $(\mathcal{I}_{\delta}, D(\mathcal{I}_{\delta}))$ be the adjoint operator to $(\mathcal{J}_{\delta}, D(\mathcal{J}_{\delta}))$.
The next lemma follows immediately by \eqref{IBP}.
\begin{Lemma}
 For each $F \in D(\mathcal{I}_{\delta})$ the action of $\mathcal{I}_{\delta}$ is given by $L_{\delta}F$, i.e. $\mathcal{I}_{\delta}F = L_{\delta}F$.
\end{Lemma}
This shows that for each $R_0 \in L^1(\Gamma_0^2, \dm \lambda)$ there exists exactly one weak solution $(R_t^{\delta})_{t \geq 0} \subset L^1(\Gamma_0^2, \dm \lambda)$ to
\begin{align}\label{POTTS:05}
 \frac{\dm }{\dm t}\int \limits_{\Gamma_0^2}F(\eta)R_{t}^{\delta}(\eta)\dm \lambda(\eta) = \int \limits_{\Gamma_0^2}L_{\delta}F(\eta)R_{t}^{\delta}(\eta)\dm \lambda(\eta), \ \ R_{t}^{\delta}|_{t=0}= R_0,
\end{align}
where $F \in D(\mathcal{I}_{\delta})$, cf. \cite{BALL77}. This solution is given by $R_t^{\delta} = S_{\delta}(t)R_0$, $t \geq 0$.

\subsubsection*{Step 2: Evolution of localized correlation functions}
Denote by $\widehat{L}_{\delta} := K_0^{-1}L_{\delta}K_0$ the operator on quasi-observables equipped with the domain $D(\widehat{L})$.
The next lemma follows by $R_{\delta} \leq 1$, a repetition of the previous arguments and, Trotter-Kato approximation.
\begin{Lemma}\label{POTTSTH:04}
 For any $\delta > 0$, the assertions of Theorem \ref{POTTSTH:00} and Theorem \ref{POTTSTH:01} hold with $z^{\pm}_{\delta}$ instead of $z^{\pm}$.
 Let $\widehat{T}_{\delta}(t)$ and $\widehat{T}_{\delta}(t)^*$ be the semigroups on $\Lb_{\alpha}$ and $\K_{\alpha}$, respectively. 
 Then, for any $G \in \Lb_{\alpha}$,
 \[
  \widehat{T}_{\delta}(t)G \longrightarrow \widehat{T}(t)G, \ \ \delta \to 0
 \]
 holds in $\Lb_{\alpha}$.
\end{Lemma}
Let $\mathcal{B}_{\alpha}^{\delta}$ be the Banach space of all equivalence classes of functions $G$ with norm
\[
 \Vert G \Vert_{\mathcal{B}_{\alpha}^{\delta}} = \int \limits_{\Gamma_0^2}|G(\eta)|e_{\lambda}(R_{\delta};\eta^+)e_{\lambda}(R_{\delta};\eta^-)e^{\alpha|\eta|}e_{\lambda}(\rho;\eta)\dm \lambda(\eta).
\]
Its dual Banach space is identified with the Banach space $\mathcal{R}_{\alpha}^{\delta}$ of all equivalence classes of functions $u$ equipped with the norm
\[
 \Vert u \Vert_{\mathcal{R}_{\alpha}^{\delta}} = \esssup \limits_{\eta \in \Gamma_0^2}\frac{|u(\eta)|}{e_{\lambda}(R_{\delta};\eta^+)e_{\lambda}(R_{\delta};\eta^-)e^{\alpha|\eta|}e_{\lambda}(\rho;\eta)}.
\]
The same arguments as in the proof of Theorem \ref{POTTSTH:00} and Theorem \ref{POTTSTH:01} show that we can replace $\Lb_{\alpha}$, $\K_{\alpha}$ 
by $\mathcal{B}_{\alpha}^{\delta}$ and $\mathcal{R}_{\alpha}^{\delta}$. Let $U_{\delta}(t)$ and $U_{\delta}(t)^*$ be the corresponding semigroups
and let $(\widehat{L}_{\delta}, D^{\mathcal{B}}(\widehat{L}_{\delta}))$ be the generator of $U_{\delta}(t)$. In analogy to $D(\widehat{L})$, we see that
\[
 D^{\mathcal{B}}(\widehat{L}_{\delta}) = \{ G \in \mathcal{B}_{\alpha}^{\delta} \ | \ M \cdot G \in \mathcal{B}_{\alpha}^{\delta} \}
\]
and, in particular, $B_{bs}(\Gamma_0^2) \subset D^{\mathcal{B}}(\widehat{L}_{\delta})$ is a core. Thus, the Cauchy problem
\[
 \frac{\dm}{\dm t}\langle G, u_t^{\delta}\rangle = \langle \widehat{L}_{\delta}G, u_t^{\delta}\rangle, \ \ u_t^{\delta}|_{t=0} = u_0, \ \ G \in B_{bs}(\Gamma_0^2)
\]
has for every $u_0 \in \mathcal{R}_{\alpha}^{\delta}$ a unique weak solution in $\mathcal{R}_{\alpha}^{\delta}$ given by $u_t^{\delta} = U_{\delta}(t)^*u_0$.
\begin{Lemma}\label{POTTSTH:02}
 Let $k_0 \in \mathcal{R}_{\alpha}^{\delta}$, then $\widehat{T}_{\delta}(t)^*k_0 = U_{\delta}(t)^*k_0$ holds.
\end{Lemma}
\begin{proof}
 Observe that $\mathcal{R}_{\alpha}^{\delta} \subset \K_{\alpha}$ is embedded continuously. Consequently, $u_t^{\delta} = U_{\delta}(t)^*k_0$
 and $k_t^{\delta} = \widehat{T}_{\delta}(t)^*k_0$ are well-defined. Since also $\Lb_{\alpha} \subset \mathcal{B}_{\alpha}^{\delta}$ is continuously embedded,
 we obtain $D(\widehat{L}) \subset D^{\mathcal{B}}(\widehat{L}_{\delta})$, i.e. $(\widehat{L}_{\delta}, D^{\mathcal{B}}(\widehat{L}_{\delta}))$
 is an extension of $(\widehat{L}_{\delta}, D(\widehat{L}))$. Therefore, $(u_t^{\delta})_{t \geq 0}$ is also a weak solution to \eqref{WCORR}
 and hence uniqueness implies $k_t^{\delta} = u_t^{\delta}$, $t \geq 0$.
\end{proof}
\begin{Lemma}\label{POTTSTH:03}
 Let $\beta^+ < \alpha^+$, $\beta^- < \alpha^-$, $k_0 \in \mathcal{R}_{\beta}^{\delta}$ and assume that $k_0$ is positive definite. 
 Then, $u_t^{\delta} := U_{\delta}(t)^*k_0$ is positive definite, for any $t \geq 0$.
\end{Lemma}
\begin{proof}
 Define a bounded linear operator $\mathcal{H}: \mathcal{R}_{\alpha}^{\delta} \longrightarrow \Lb_{c}$, for any $c = (c^+,c^-) \in \R^2$, by
 \[
  \mathcal{H}u(\eta) := \int \limits_{\Gamma_0^2}(-1)^{|\xi|}u(\eta \cup \xi)\dm \lambda(\xi).
 \]
 Let $G \in \mathcal{B}_{\alpha}^{\delta}$ be arbitrary. Then, for any $u \in \mathcal{R}_{\alpha}^{\delta}$, we get by Fubini's theorem and \eqref{IBP}
 \begin{align}\label{GMCS:32}
  \langle K_0G, \mathcal{H}u \rangle = \langle G, u\rangle.
 \end{align}
 We can apply Fubini's theorem and \eqref{IBP} since
 \begin{align*}
  &\ \int \limits_{\Gamma_0^2}\int \limits_{\Gamma_0^2}\int \limits_{\Gamma_0^2}|G(\xi)||u(\eta \cup \xi \cup \zeta)|\dm \lambda(\zeta)\dm \lambda(\xi)\dm \lambda(\eta)
  \\ &\leq \Vert u \Vert_{\mathcal{R}_{\alpha}^{\delta}}e^{2 e^{\alpha^+}\langle R_{\delta}\rangle_{\rho} } e^{2 e^{\alpha^-}\langle R_{\delta} \rangle_{\rho}} \int \limits_{\Gamma_0^2}|G(\xi)|e^{\alpha|\xi|} e_{\lambda}(\rho;\xi) e_{\lambda}(R_{\delta};\xi^+)e_{\lambda}(R_{\delta};\xi^-)\dm \lambda(\xi)
 \end{align*}
 is satisfied, where $\langle R_{\delta} \rangle_{\rho} := \int \limits_{\R^d}R_{\delta}(x)\rho(x)\dm x$.
 For the same $u$ and $G \in D^{\mathcal{B}}(\widehat{L})$ we obtain by \eqref{GMCS:32} and $K_{0}\widehat{L}_{\delta}G = L_{\delta}K_0G$ 
 \begin{align}\label{GMCS:33}
  \langle \widehat{L}_{\delta}G, u \rangle = \langle K_0\widehat{L}_{\delta}G, \mathcal{H}u \rangle = \langle L_{\delta}K_0G, \mathcal{H}u\rangle.
 \end{align}
 Observe that
 \begin{align*}
  \langle G,u^{\delta}_t\rangle = \langle G, u_0 \rangle + \int \limits_{0}^{t}\langle \widehat{L}_{\delta}G, u^{\delta}_s \rangle \dm s, \ \ G \in D^{\mathcal{B}}(\widehat{L}).
 \end{align*}
 Let $G \in \K_{c}$, where $c := (\log(2),\log(2))$ . Then, by $M(\eta) \leq 2 |\eta|$, 
 \begin{align*}
  & \int \limits_{\Gamma_0^2}M(\eta)|G(\eta)|e^{\alpha|\eta|}e_{\lambda}(\rho;\eta)e_{\lambda}(R_{\delta};\eta^+)e_{\lambda}(R_{\delta};\eta^-)\dm \lambda(\eta)
  \\ &\leq 2\Vert G \Vert_{\K_{c}} \int \limits_{\Gamma_0^2}|\eta|2^{|\eta|}e^{\alpha|\eta|}e_{\lambda}(\rho;\eta)e_{\lambda}(R_{\delta};\eta^+)e_{\lambda}(R_{\delta};\eta^-)\dm \lambda(\eta)
  \\ &= 2 \Vert G \Vert_{\K_{c}} \sum \limits_{n,m=0}^{\infty}\frac{1}{n!}\frac{1}{m!}(n+m)e^{\alpha^+ n}e^{\alpha^- m} \langle R_{\delta}\rangle_{\rho}^{n+m} < \infty.
 \end{align*}
 This implies $\mathcal{K}_{c} \subset D^{\mathcal{B}}(\widehat{L})$. 
 By \eqref{GMCS:32} and \eqref{GMCS:33} it follows for $R^{\delta}_t := \mathcal{H}u^{\delta}_t \in L^1(\Gamma_0^2, \dm \lambda)$, $t \geq 0$,
 \begin{align}\label{POTTS:11}
  \langle K_0G, R^{\delta}_t \rangle = \langle K_0 G, R_0 \rangle + \int \limits_{0}^{t}\langle L_{\delta}K_0 G, R^{\delta}_s\rangle \dm s, \ \ G \in \K_{c}.
 \end{align}
 For any $F \in D(\mathcal{J}_{\delta}) \subset L^{\infty}(\Gamma_0^2, \dm \lambda)$ we get $|K_0^{-1}F(\eta)| \leq \Vert F \Vert_{L^{\infty}}2^{|\eta|}$
 and hence $D(\mathcal{J}_{\delta}) \subset K_0\K_{c}$. Thus, we can find $G \in \K_{c}$ such that $K_0G = F$.  By \eqref{POTTS:11}, it follows that
 \[
  \langle F, R^{\delta}_t \rangle = \langle F, R_0 \rangle + \int \limits_{0}^{t}\langle \mathcal{J}_{\delta}F, R^{\delta}_s \rangle \dm s, \ \ F \in D(\mathcal{J}_{\delta}).
 \]
 Recall that $k_0 \in \mathcal{R}_{\beta}^{\delta}$, hence, by Theorem \ref{POTTSTH:01}, we see that $u_t^{\delta} = U_{\delta}(t)^*k_0$ is continuous in $t \geq 0$ 
 w.r.t. the norm in $\mathcal{R}_{\alpha}^{\delta}$. Since $\mathcal{H}: \mathcal{R}_{\alpha}^{\delta} \longrightarrow L^1(\Gamma_0^2, \dm \lambda)$ is continuous, 
 $R_t^{\delta} = \mathcal{H}u_t^{\delta}$ is continuous in $t \geq 0$ w.r.t. the norm in $L^1(\Gamma_0^2, \dm \lambda)$. 
 Hence, $(R_t^{\delta})_{t \geq 0}$ is a weak solution to \eqref{POTTS:05}. Uniqueness implies that $R_t^{\delta} = S_{\delta}(t)R_0 \geq 0$.
 Finally, for any $G \in B_{bs}^+(\Gamma_0^2)$ we get
 \[
  \langle G, u_t^{\delta}\rangle = \langle K_0G, R_t^{\delta}\rangle \geq 0, \ \ t \geq 0.
 \]
\end{proof}

\subsubsection*{Step 3: Proof of Proposition \ref{POTTSTH:05}}
Let $\beta^+ < \alpha^+$, $\beta^- < \alpha^-$. First, we consider the special case $\mu_0 \in \mathcal{P}_{\beta}$. Let $k_0 \in \K_{\beta}$ be the associated
correlation function. Define
\[
 k_{0,\delta}(\eta) := k_0(\eta)e_{\lambda}(R_{\delta};\eta^+)e_{\lambda}(R_{\delta};\eta^-), \ \ \delta > 0, \ \eta \in \Gamma_0^2,
\]
then $k_{0,\delta} \in \mathcal{R}_{\beta}^{\delta}$. Moreover, $k_{0,\delta}$ is positive definite, cf. \cite{F11TC}. 
Lemma \ref{POTTSTH:02} implies $\widehat{T}_{\delta}(t)^*k_{0,\delta} = U_{\delta}(t)^*k_{0,\delta} \in \mathcal{R}_{\alpha}^{\delta}$ and 
Lemma \ref{POTTSTH:03} shows that $U_{\delta}(t)^*k_{0,\delta}$ is positive definite. Let $G \in B_{bs}^+(\Gamma_0^2)$, then
$\langle G, \widehat{T}_{\delta}(t)^* k_{0,\delta}\rangle \geq 0$. Observe that
\begin{align}\label{POTTS:06}
 \langle G, \widehat{T}_{\delta}(t)^*k_{0,\delta}\rangle = \langle \widehat{T}_{\delta}(t)^*G - \widehat{T}(t)G, k_{0,\delta}\rangle + \langle \widehat{T}(t)G, k_{0,\delta}\rangle.
\end{align}
For the first term we obtain, by $\Vert k_{0,\delta} \Vert_{\K_{\alpha}} \leq \Vert k_0 \Vert_{\K_{\alpha}}$,
\begin{align*}
 |\langle \widehat{T}_{\delta}(t)^*G - \widehat{T}(t)G, k_{0,\delta}\rangle| 
 \leq \Vert \widehat{T}_{\delta}(t)G - \widehat{T}(t)G\Vert_{\Lb_{\alpha}}\Vert k_0 \Vert_{\K_{\alpha}}.
\end{align*}
The latter tends to zero, see Lemma \ref{POTTSTH:04}. The second term in \eqref{POTTS:06} tends, by dominated convergence, 
to $\langle \widehat{T}(t)G, k_0\rangle = \langle G, \widehat{T}(t)^*k_0\rangle$. Thus
\[
 \langle G, \widehat{T}_{\delta}(t)^* k_{0,\delta}\rangle \longrightarrow \langle G, \widehat{T}(t)^*k_0 \rangle, \ \ \delta \to 0
\]
and hence $\langle G, \widehat{T}(t)^*k_0 \rangle \geq 0$, i.e. $\widehat{T}(t)^*k_0$ is positive definite.

For the general case, let $\mu_0 \in \mathcal{P}_{\alpha}$ with correlation function $k_0 \in \K_{\alpha}$.
Then $k_{0, \delta}(\eta) := e^{-\delta|\eta|} k_0(\eta)$ belongs to $\K_{\alpha - \delta}$ for any $\delta > 0$. 
By previous case, we see that $\widehat{T}(t)^* k_{0,\delta} \in \K_{\alpha}$ is positive definite. Taking the limit $\delta \to 0$ yields the assertion.

\section{Ergodicity}
Suppose the same conditions as for the previous section. The following is the main statement for this section.
\begin{Proposition}\label{PROP:00}
 There exists a unique invariant measure $\mu_{\mathrm{inv}} \in \mathcal{P}_{\alpha}$ associated to $L$, i.e.
 \begin{align}\label{POTTS:08}
  \int \limits_{\Gamma^2}LF(\gamma)\dm \mu_{\mathrm{inv}}(\gamma) = 0, \ \ F \in \mathcal{FP}(\Gamma^2).
 \end{align}
 Let $k_{\mathrm{inv}}$ be the associated correlation function.
 \begin{enumerate}
  \item The semigroup $\widehat{T}(t)$ is uniformly ergodic with exponential rate and the projection operator is given by
  \begin{align}\label{ERGOD:01}
   \widehat{P}G(\eta) = \int \limits_{\Gamma_0^2}G(\xi)k_{\mathrm{inv}}(\xi)\dm \lambda(\xi) \1^*(\eta).
  \end{align}
  \item The adjoint semigroup $\widehat{T}(t)^*$ is uniformly ergodic with exponential rate and the projection operator is given by
  \begin{align}\label{ERGOD:00}
   \widehat{P}^*k(\eta) = k_{\mathrm{inv}}(\eta)k(\emptyset).
  \end{align}
  \item There exists constants $a,b > 0$ such that for all $\mu_0 \in \mathcal{P}_{\alpha}$
  \[
   d_{\alpha}(\mu_t, \mu_{\mathrm{inv}}) \leq a e^{-b t} d(\mu_0, \mu_{\mathrm{inv}}), \ \ t \geq 0
  \]
  holds, where $(\mu_t)_{t \geq 0}$ is the unique weak solution to \eqref{FPE}. 
 \end{enumerate}
\end{Proposition}
The rest of this section is devoted to the proof of this proposition.
Multiplication by $\1^*$ and $1 - \1^*$ defines projection operators $\1^*: \K_{\alpha} \longrightarrow \K_{\alpha}^0$ and
$(1-\1^*): \K_{\alpha} \longrightarrow \K_{\alpha}^{\geq 1}$, respectively. Here, $\K^{\geq 1}_{\alpha} = \{ k \in \K_{\alpha} \ | \ k^{(0,0)} = 0 \}$
and $\K_{\alpha}^{0} = \{ k \in \K_{\alpha} \ | \ k^{(n,m)} = 0, \ \ n+m \geq 1\}$.
By $\1^* (1-\1^*) = (1-\1^*)\1^* = 0$ we obtain $\K_{\alpha} = \K_{\alpha}^0 \oplus \K_{\alpha}^{\geq 1}$.
Define a linear operator $S$ by $Sk(\emptyset) = 0$ and
\[
 Sk(\eta) = \frac{1}{M(\eta)}L^{\Delta}k(\eta) + k(\eta), \ \ \eta \neq \emptyset.
\]
It is not difficult to see that $S$ leaves $\K_{\alpha}^{\geq 1}$ invariant and $\Vert S \Vert_{L(\K_{\alpha})} < 1$.
The next lemma provides existence and uniqueness of solutions to $L^{\Delta}k = 0$. Its proof is an easy modification of the arguments in \cite{FKK12}.
\begin{Lemma}
 The equation
 \begin{align}\label{GMCS:02}
  L^{\Delta}k_{\mathrm{inv}} = 0, \ \ k_{\mathrm{inv}}(\emptyset,\emptyset) = 1
 \end{align}
 has a unique solution $k_{\mathrm{inv}} \in \K_{\alpha}$. This solution is given by
 $k_{\mathrm{inv}} = \1^* + (1 - S)^{-1}S\1^*$, where
 \begin{align*}
  S\1^*(\eta) = \1_{\Gamma_0^{(1)}}(\eta^+)0^{|\eta^-|}z^+ + \1_{\Gamma_0^{(1)}}(\eta^-)0^{|\eta^+|}z^-.
 \end{align*}
 In particular, \eqref{ERGOD:00} is a projection operator on $\K_{\alpha}$ with range 
 \[
  \mathrm{Ran}(\widehat{P}^*) = \{ k \in D(L^{\Delta}) \ | \ L^{\Delta}k = 0 \}
 \]
 and it is given by $\widehat{P}^* = \1^* + (1 - S)^{-1}S\1^*$, where $\1^*$ acts as a multiplication operator.
\end{Lemma}
First we establish ergodicity for $\widehat{T}(t)$. Let $\Lb_{\alpha}^{0} := \{ G \in \Lb_{\alpha} \ | \ G = \kappa \1^*, \ \kappa \in \R\}$ and
$\Lb_{\alpha}^{\geq 1} := \{ G \in \Lb_{\alpha} \ | \ G(\emptyset) = 0 \}$.
Then $\Lb_{\alpha} = \Lb_{\alpha}^{0} \oplus \Lb_{\alpha}^{\geq 1}$ and the projection onto $\Lb_{\alpha}^0$ is given by multiplication
with $\1^*$. Likewise, $1 - \1^*$ projects onto $\Lb_{\alpha}^{\geq 1}$.
Define $B_{01}: \Lb_{\alpha}^{\geq 1} \longrightarrow \Lb_{\alpha}^{0}$, $B_{01}G = \1^*BG$ and 
$L_{11}: \Lb_{\alpha}^{\geq 1} \longrightarrow \Lb_{\alpha}^{\geq 1}$, $L_{11}G = AG + (1-\1^*)BG$. Taking into account $\widehat{L} = \widehat{L}(1-\1^*)$ yields
\begin{align}\label{GMCS:38}
 \widehat{L}G = B_{01}(1-\1^*)G + L_{11}(1-\1^*)G, \ \ G \in \Lb_{\alpha}.
\end{align}
Moreover, since $D(\widehat{L}) = \{ G \in \Lb_{\alpha}\ | \ M\cdot G \in \Lb_{\alpha}\}$ and $\Lb_{\alpha}^{0} \subset D(\widehat{L})$
it follows that $D(L_{11}) = D(\widehat{L})\cap \Lb_{\alpha}^{\geq 1}$. 
Note that $B_{01}$ is given by
\[
 \1^*BG(\eta) = \1^*(\eta)z^-\int \limits_{\R^d}G(\emptyset,x)\dm x + \1^*(\eta)z^+\int \limits_{\R^d}G(x,\emptyset)\dm x
\]
and hence is a positive operator. The next statement was shown for the one-component Glauber dynamics in \cite{KKM10}.
\begin{Theorem}
 Let $a(\alpha) \in (0,1)$ be given by \eqref{POTTS:09} and 
 \begin{align}\label{GMCS:37}
  \omega_0 := \sup \left\{ \omega \in \left[0, \frac{\pi}{4}\right] \ \bigg | \ a(\alpha) < \cos(\omega) \right\}.
 \end{align}
 Then the following statements hold:
 \begin{enumerate}
  \item The point $0$ is an eigenvalue for $(\widehat{L}, D(\widehat{L}))$ with eigenspace $\Lb_{\alpha}^{0}$ and eigenvector $\1^*$.
  \item Let $\lambda_0 := (1 - a(\alpha)) > 0$. Then
  \[
   I_1 := \{ \lambda \in \C \ | \ \mathrm{Re}(\lambda) > - \lambda_0 \} \backslash \{0\}
  \]
  and
  \[
   I_2 := \left\{ \lambda \in \C \ \bigg| \ |\mathrm{arg}(\lambda)| < \frac{\pi}{2} + \omega_0 \right\} \backslash \{0\}
  \]
  belong to the resolvent set $\rho(\widehat{L})$ of $\widehat{L}$ on $\Lb_{\alpha}$.
 \end{enumerate}
\end{Theorem}
\begin{proof}
 Let $(A_1, D(L_{11}))$ be the restriction of $(A,D(\widehat{L}))$ to $\Lb_{\alpha}^{\geq 1}$ and denote by $\Vert \cdot \Vert_{\Lb_{\alpha}^{\geq 1}}$
 the norm on $\Lb_{\alpha}^{\geq 1}$. Observe that $M(\eta) \geq 1$ for all $|\eta| \geq 1$. 
 Then, for any $\lambda = u + i w$, $u \geq 0$, $w \in \R$, by $M(\eta) \geq 1$ for all $|\eta| \geq 1$,
 \[
  \left| \frac{G}{\lambda + M(\eta)} \right| \leq \frac{|G|}{\sqrt{ (u + 1)^2 + w^2 }} \leq |G| \min\left( \frac{1}{|\lambda|}, \frac{1}{\sqrt{1 + w^2}}\right).
 \]
 This implies $\lambda \in \rho(A_1)$ and
 \begin{align}\label{GMCS:35}
  \Vert R(\lambda;A_1)G \Vert_{\Lb_{\alpha}^{\geq 1}} \leq \min\left( \frac{1}{|\lambda|}, \frac{1}{\sqrt{1 + w^2}}\right)\Vert G \Vert_{\Lb_{\alpha}^{\geq 1}}.
 \end{align} 
 Consider the decomposition
 \begin{align}\label{GMCS:36}
  (\lambda - L_{11}) = (1 - (1-\1^*)BR(\lambda; A_1))(\lambda - A_1).
 \end{align}
 Then, by \eqref{POTTS:09},
 \begin{align*}
  \Vert (1 - \1^*)BG \Vert_{\Lb_{\alpha}^{\geq 1}} 
  \leq \int \limits_{\Gamma_0^2} |BG(\eta)|e^{\alpha|\eta|}e_{\lambda}(\rho;\eta)\dm \lambda(\eta)
  \leq a(\alpha) \Vert M\cdot G \Vert_{\Lb_{\alpha}^{\geq 1}},
 \end{align*}
 for any $G \in \Lb_{\alpha}^{\geq 1}$. This implies that $(1 - (1-\1^*)BR(\lambda;A_1))$ is invertible on $\Lb_{\alpha}^{\geq 1}$, i.e. $\lambda \in \rho(L_{11})$, and
 \begin{align}\label{GMCS:34}
  R(\lambda; L_{11}) = R(\lambda;A_1) (1 - (1-\1^*)BR(\lambda; A_1))^{-1}.
 \end{align}
 In particular, we obtain for $\lambda = u + iw$, $u \geq 0$, $w \in \R$ by \eqref{GMCS:34} and \eqref{GMCS:35}
 \[
  \Vert R(\lambda;L_{11})G \Vert_{\Lb_{\alpha}^{\geq 1}} \leq \frac{ \min\left( \frac{1}{|\lambda|}, \frac{1}{\sqrt{1 + w^2}}\right)}{1 - a(\alpha)}\Vert G \Vert_{\Lb_{\alpha}^{\geq 1}}
 \]
 and for $\lambda = iw$, $w \in \R$
 \[
  \Vert R(iw, L_{11})G \Vert_{\Lb_{\alpha}^{\geq 1}} \leq \frac{\sqrt{1 + w^2}^{-1}}{1 - a(\alpha)}\Vert G \Vert_{\Lb_{\alpha}^{\geq 1}}.
 \]
 For $\lambda = u + iw$, $0 > u > - \lambda_0$ and $w \in \R$ write
 \[
  ( u +iw - L_{11})= (1 + u R(iw;L_{11}))(iw - L_{11}).
 \]
 Then, by $|u| < \lambda_0$ and $\frac{|u|}{\sqrt{1 + w^2}}\frac{1}{1-a} \leq \frac{|u|}{\lambda_0} < 1$ we obtain $\lambda \in \rho(L_{11})$ and
 \[
  \Vert R(\lambda; L_{11}) G \Vert_{\Lb_{\alpha}^{\geq 1}} \leq \frac{\sqrt{1 + w^2}^{-1}}{1 - a(\alpha)} \left(1 - \frac{|u|}{\lambda_0}\right)^{-1}\Vert G \Vert_{\Lb_{\alpha}^{\geq 1}}.
 \]
 Therefore, $I_1$ belongs to the resolvent set of $L_{11}$. 
 For $I_2$ let $\lambda = u + iw \in I_2$ and $u < 0$. Then, there exists $\omega \in (0,\omega_0)$ such that $|\mathrm{arg}(\lambda)| < \frac{\pi}{2} + \omega$ and hence
 \[
  |w| = |\tan(\arg(\lambda))| |u| \geq \cot(\omega) |u|.
 \]
 This implies for $\eta \neq \emptyset$
 \[
  |\lambda + M(\eta)|^2 = (u + M(\eta))^2 + w^2 \geq  (u + M(\eta))^2 + \cot(\omega)^2 u^2.
 \]
 The right-hand side is minimal for the choice $u = - \frac{M(\eta)}{1 + \cot(\omega)^2}$ which yields
 \begin{align*}
  |\lambda + M(\eta)|^2 &\geq M(\eta)^2\left( \left( \frac{\cot(\omega)^2}{ 1 + \cot(\omega)^2}\right)^2 + \frac{\cot(\omega)^2}{(1+\cot(\omega)^2)^2}\right)
  \\ &= M(\eta)^2 \frac{\cot(\omega)^2}{1+\cot(\omega)^2} = M(\eta)^2 \cos(\omega)^2.
 \end{align*}
 Then, by 
 \[
  \Vert (1-\1^*)BR(\lambda;A_1)G\Vert_{\Lb_{\alpha}^{\geq 1}} \leq a(\alpha) \Vert A_1 R(\lambda;A_1)G\Vert_{\Lb_{\alpha}^{\geq 1}} \leq \frac{a(\alpha)}{\cos(\omega)}\Vert G \Vert_{\Lb_{\alpha}^{\geq 1}} 
 \]
 and \eqref{GMCS:37} we have $a(\alpha) < \cos(\omega)$.
 By \eqref{GMCS:36} we obtain $I_2 \subset \rho(L_{11})$. 
 Moreover, for each $\lambda = u + iw$ such that $\frac{\pi}{2} < |\mathrm{arg}(\lambda)| < \frac{\pi}{2} + \omega$ and, for some $\omega \in (0,\omega_0)$,
 \begin{align*}
  \Vert R(\lambda;L_{11})G \Vert_{\Lb_{\alpha}^{\geq 1}} &\leq \frac{\sqrt{(u^2 + 1)^2 + w^2}^{-1}}{1 - \frac{a(\alpha)}{\cos(\omega)}} \Vert G \Vert_{\Lb_{\alpha}^{\geq 1}}
  \\ &\leq \frac{(1 - \frac{a(\alpha)}{\cos(\omega)})^{-1}}{|w|} \Vert G \Vert_{\Lb_{\alpha}^{\geq 1}}
  \leq \sqrt{2}\frac{(1 - \frac{a(\alpha)}{\cos(\omega)})^{-1}}{|\lambda|} \Vert G \Vert_{\Lb_{\alpha}^{\geq 1}},
 \end{align*}
 where we have used $|w| \geq \frac{|\lambda|}{\sqrt{2}}$.
 For the first claim let $\psi \in D(\widehat{L})$ be an eigenvector to the eigenvalue $0$. The decomposition $\psi = \1^* \psi + (1 - \1^*)\psi = \psi_0 + \psi_1$ with
 $\psi_0 \in \Lb_{\alpha}^{0}$ and $\psi_1 \in \Lb_{\alpha}^{\geq 1} \cap D(\widehat{L}) = D(L_{11})$ yields, by \eqref{GMCS:38},
 \[
  0 = \widehat{L}\psi = \1^* B\psi_1 + L_{11}\psi_1 \in \Lb_{\alpha}^0 \oplus \Lb_{\alpha}^{\geq 1}.
 \]
 Hence $L_{11}\psi_1 = 0$ and since $0 \in \rho(L_{11})$ also $\psi_1 = 0$.
 For the second statement let $\lambda \in I_1 \cup I_2$ and $H = H_0 + H_1 \in \Lb_{\alpha}^0 \oplus \Lb_{\alpha}^{\geq 1}$. 
 Then, we have to find $G \in D(\widehat{L})$ such that
 \[
  (\lambda - \widehat{L})G = H.
 \]
 Using again the decomposition of $\widehat{L}$, above equation is equivalent to the system of equations
 \begin{align*}
  \lambda G_0 - \1^* BG_1 &= H_0
  \\ (\lambda - L_{11})G_1 &= H_1.
 \end{align*}
 Since $\lambda \in I_1 \cup I_2 \subset \rho(L_{11})$ the second equation has a unique solution on $\Lb_{\alpha}^{\geq 1}$ given by $G_1 = R(\lambda;L_{11})H_1$.
 Therefore, $G_0$ is given by
 \[
  G_0 = \frac{1}{\lambda}\left( H_0 + \1^* BR(\lambda; L_{11})H_1\right).
 \]
\end{proof}
\begin{Remark}\label{GMCSRK:01}
 The proof shows that for any $\e > 0$ there exists $\omega = \omega(\e) \in (0, \frac{\pi}{2})$ such that
 \[
  \Sigma(\e) := \left\{ \lambda \in \C \ \bigg| \ |\mathrm{arg}(\lambda + \lambda_0 - \e)| \leq \frac{\pi}{2} + \omega \right\} \subset I_1 \cup I_2 \cup \{0\}
 \]
 and there exists $M(\e) > 0$ with
 \[
  \Vert R(\lambda; L_{11})G\Vert_{\Lb_{\alpha}^{\geq 1}} \leq \frac{M(\e)}{|\lambda|} \Vert G \Vert_{\Lb_{\alpha}^{\geq 1}}
 \]
 for all $\lambda \in \Sigma(\e) \backslash \{0\}$. Moreover, $(L_{11}, D(L_{11}))$ is a sectorial operator of angle $\omega_0$ on $\Lb_{\alpha}^{\geq 1}$.
 Denote by $\widetilde{T}(t)$ the bounded analytic semigroup on $\Lb_{\alpha}^{\geq 1}$ given by
 \begin{align}\label{GMCS:53}
  \widetilde{T}(t) = \frac{1}{2\pi i} \int \limits_{\sigma} e^{\zeta t}R(\zeta;L_{11}) \dm \zeta, \ \ t > 0,
 \end{align}
 where the integral converges in the uniform operator topology, see \cite{PAZ83}. Here, $\sigma$ denotes any piecewise smooth curve in
 \[
  \left\{ \lambda \in \C \ \bigg| \ |\mathrm{arg}(\lambda)| < \frac{\pi}{2}+\omega_0 \right\} \backslash \{0\}
 \]
 running from $\infty e^{-i\theta}$ to $\infty e^{i \theta}$ for $\theta \in (\frac{\pi}{2}, \frac{\pi}{2}+ \omega_0)$.
\end{Remark}
The $\Lb_{\alpha}^{\geq 1}$ part of $\widehat{T}(t)$ is given by $(1 - \1^*)\widehat{T}(t)|_{\Lb_{\alpha}^{\geq 1}}$ and hence has the generator
$(1 - \1^*)\widehat{L}G = L_{11}G$. As a consequence, we obtain $\widetilde{T}(t) = (1 - \1^*)\widehat{T}(t)|_{\Lb_{\alpha}^{\geq 1}}$.
This yields the decomposition
\begin{align}\label{GMCS:67}
 \widehat{T}(t) = \1^* + \1^*\widehat{T}(t)(1-\1^*) + \widetilde{T}(t)(1-\1^*), \ \ t \geq 0.
\end{align}
By duality we see that the adjoint semigroup $(\widehat{T}(t)^*)_{t \geq 0}$ on $\K_{\alpha}$ admits the decomposition
\begin{align}\label{GMCS:68}
 \widehat{T}(t)^* = \1^* + (1-\1^*)\widehat{T}(t)^*\1^* + \widetilde{T}(t)^*(1-\1^*), \ \ t \geq 0,
\end{align}
where $\widetilde{T}(t)^* \in L(\K_{\alpha}^{\geq 1})$ is the adjoint semigroup to $(\widetilde{T}(t))_{t \geq 0}$.
\begin{Lemma}
 The projection operator $\widehat{P}: \Lb_{\alpha} \longrightarrow \Lb_{\alpha}^0$, given by \eqref{ERGOD:01}, satisfies
 \[
  \langle \widehat{P}G, k \rangle = \langle G, \widehat{P}^*k\rangle.
 \]
 Moreover, we have $\widehat{P} = \widehat{T}(t)\widehat{P} = \widehat{P}\widehat{T}(t)$ and 
 \begin{align}\label{GMCS:69}
  \widehat{T}(t)^*\widehat{P}^* = \widehat{P}^*\widehat{T}(t)^* = \widehat{P}^*.
 \end{align}
\end{Lemma}
Now we are prepared to prove Proposition \ref{PROP:00}.
\begin{proof}(Proposition \ref{PROP:00})\\
 The spectral properties stated in Remark \ref{GMCSRK:01}, formulas \eqref{GMCS:53}, \eqref{GMCS:67} and, \eqref{GMCS:68}
 imply that for any $\e > 0$ there exists $C(\e) > 0$ such that
 \[
  \Vert (1-\1^*)\widehat{T}(t)G\Vert_{\Lb_{\alpha}} \leq C(\e)e^{-(\lambda_0 - \e)t}\Vert G \Vert_{\Lb_{\alpha}}, \ \ G \in \Lb_{\alpha}^{\geq 1}, \ \ t \geq 0.
 \]
 Repeat, e.g., the arguments in \cite{KKM10}. This yields 
 \[
  \Vert \widehat{T}(t)^*k \Vert_{\K_{\alpha}} \leq C(\e)e^{-(\lambda_0 - \e)t}\Vert k \Vert_{\K_{\alpha}}, \ \ k \in \K_{\alpha}^{\geq 1}.
 \]
 Let $k \in \K_{\alpha}$, we obtain, by \eqref{ERGOD:00},
 \[
  k - \widehat{P}^*k = (1-\1^*)k \cdot k_{\mathrm{inv}} \in \K_{\alpha}^{\geq 1}.
 \]
 Using \eqref{GMCS:69}, we see that
 \begin{align}\label{GMCS:64}
  \Vert \widehat{T}(t)^*k - \widehat{P}^*k \Vert_{\K_{\alpha}} = \Vert \widehat{T}(t)^*(k - \widehat{P}^*k)\Vert_{\K_{\alpha}}
  \leq C(\e)e^{-(\lambda_0 - \e)t}\Vert k - \widehat{P}^*k \Vert_{\K_{\alpha}}
 \end{align}
 holds. This shows that $\widehat{T}(t)^*$ is uniformly ergodic with exponential rate. By duality also $\widehat{T}(t)$ is uniformly ergodic with exponential rate.
 Let $\mu_0 \in \mathcal{P}_{\alpha}$ and $\mu_t \in \mathcal{P}_{\alpha}$ be the weak solution to \eqref{FPE}. Denote by $(k_{\mu_t})_{t \geq 0} \subset \K_{\alpha}$ 
 its associated family of correlation functions. Then, for any $t \geq 0$,
 \begin{align*}
  \Vert k_{\mu_t} - k_{\mathrm{inv}} \Vert_{\K_{\alpha}} 
  \leq C(\e)e^{-(\lambda_0 - \e)t} \Vert k_{\mu_0} - k_{\mathrm{inv}} \Vert_{\K_{\alpha}}
 \end{align*}
 shows that $k_{\mathrm{inv}}$ is a limit of positive definite functions. Hence, it is positive definite.
 Thus, there exists a unique measure $\mu_{\mathrm{inv}} \in \mathcal{P}_{\alpha}$ having $k_{\mathrm{inv}}$ as its correlation function. 
 Since $\widehat{T}(t)^*k_{\mathrm{inv}} = k_{\mathrm{inv}}$, it follows that $\mu_{\mathrm{inv}}$ is invariant for $L$. 
 Property \eqref{POTTS:08} follows immediately from $L^{\Delta}k_{\mathrm{inv}} = 0$.
\end{proof}

\section{Mesoscopic limit}
The general approach to mesoscopic limits and particular examples can be found in \cite{FKK10, FFHKKK15} and the references therein. 
There are several scalings which can be used to obtain the mesoscopic limit and related kinetic equations. In this work we use the so-called Vlasov scaling.
For convenience of the reader, we give a brief description.
The aim is to produce a certain scaling $L \longmapsto L_n$, $n \in \N$ such that the following scheme holds.
Let $L_n^{\Delta}$ be the scaled operator on correlation functions and $e^{tL_{n}^{\Delta}}$ the (heuristic) representation of the scaled evolution of correlation functions.
The particular choice of $L \longrightarrow L_{n}$ should preserve the order of singularity. 
Namely, for $\beta > 0$ let $R_{\beta}G(\eta) := \beta^{|\eta|}G(\eta)$, then
\begin{align}\label{MCSL:11}
 R_{n^{-1}}e^{tL_{n}^{\Delta}}R_{n}k \longrightarrow T_V^{\Delta}(t)k, \ \ n \to \infty
\end{align}
should exist. The evolution $T_V^{\Delta}(t)$ should preserve Lebesgue-Poisson exponentials, 
i.e. if $r_0(\eta) = e_{\lambda}(\rho_0^-, \eta^-)e_{\lambda}(\rho_0^{+};\eta^+)$, then
$T_V^{\Delta}(t)r_0(\eta) = e_{\lambda}(\rho_t^{-},\eta^-)e_{\lambda}(\rho_t^{+};\eta^+)$. 
We will show that $\rho_t^{-}, \rho_t^{+}$ satisfy a certain system of non-linear integro-differential equations.

Instead of investigating the limit \eqref{MCSL:11}, observe that formally 
\[
 R_{n^{-1}}e^{tL_{n}^{\Delta}}R_{n} = e^{tR_{n^{-1}}L_n^{\Delta}R_{n}} 
\]
Thus, it is the same to consider renormalized operators $L_{n,\mathrm{ren}}^{\Delta} := R_{n^{-1}}L_n^{\Delta}R_{n}$ 
and study the behaviour of the renormalized semigroups $T_{n,\mathrm{ren}}^{\Delta}(t) := e^{tL_{n,\mathrm{ren}}^{\Delta}}$, as $n \to \infty$.
We will prove that the limit
\begin{align}\label{MCSL:12}
 L_{n,\mathrm{ren}}^{\Delta}\longrightarrow L_V^{\Delta}
\end{align}
exists and $L_V^{\Delta}$ is the generator of a semigroup $T_V^{\Delta}(t) = e^{tL_V^{\Delta}}$. 
The limit \eqref{MCSL:11} is then obtained by showing the convergence
\begin{align}\label{MCSL:13}
 T_{n,\mathrm{ren}}^{\Delta}(t) \longrightarrow T_V^{\Delta}(t)
\end{align}
in a proper sense.

Note that $L_{n,\mathrm{ren}}^{\Delta}$ and $L_V^{\Delta}$ are operators on $\K_{\alpha}$ and therefore cannot be generators of strongly continuous semigroups.
Hence, we consider first the scaled evolution on quasi-observables $\widehat{L}_n := K^{-1}L_n K$ and 
the renormalized operators $\widehat{L}_{n,\mathrm{ren}} = R_{n}\widehat{L}_nR_{n^{-1}}$.
We show that $\widehat{L}_{n,\mathrm{ren}}$ is the generator of an analytic semigroup $\widehat{T}_{n,\mathrm{ren}}(t)$ of contractions and prove that 
$\widehat{L}_{n,\mathrm{ren}} \longrightarrow \widehat{L}_V$, as $n \to \infty$. Here, $\widehat{L}_V$ is again the generator of an analytic semigroup
$\widehat{T}_V(t)$ of contractions. By Trotter-Kato approximation, it follows that $\widehat{T}_{n,\mathrm{ren}}(t) \longrightarrow \widehat{T}_V(t)$
strongly in $\Lb_{\alpha}$. By duality we obtain \eqref{MCSL:12} and \eqref{MCSL:13}.

\subsection{Assumptions}
Suppose that $\phi^{\pm}, \psi^{\pm}, \kappa^{\pm}, \tau^{\pm} \geq 0$ are symmetric, bounded and there exists $\alpha = (\alpha^+, \alpha^-) \in \R^2$
and $\rho \geq 1$ measurable and locally bounded such that the conditions below are satisfied.
\begin{enumerate}
 \item There exists $h: \R^d \longrightarrow \R_+$ with $h \cdot \rho \in L^1(\R^d)$ such that
  \[
   \sup \limits_{y \in \R^d}\ \frac{g(x-y)}{h(y)} =: C_x < \infty, \ \ x \in \R^d
  \]
 for all $g \in \{ \phi^{\pm}, \psi^{\pm}, \kappa^{\pm}, \tau^{\pm}\}$.
 \item We have
  \begin{align*}
   \sup \limits_{x \in \R^d}\ e^{-\alpha^+}e^{e^{\alpha^+}(\phi^+\ast \rho)(x)} e^{e^{\alpha^-}(\psi^+ \ast \rho)(x)} + e^{e^{\alpha^-}(\kappa^- \ast \rho)(x)}e^{e^{\alpha^+}(\tau^- \ast \rho)(x)} &< 2
   \\ \sup \limits_{x \in \R^d}\ e^{-\alpha^-}e^{e^{\alpha^-}(\phi^- \ast \rho)(x)} e^{e^{\alpha^+}(\psi^- \ast \rho)(x)} + e^{e^{\alpha^+}(\kappa^+ \ast \rho)(x)} e^{e^{\alpha^-}(\tau^+ \ast \rho)(x)} &< 2
   \\ \sup \limits_{x \in \R^d}\ e^{e^{\alpha^+}(\kappa^+ \ast \rho)(x)} e^{e^{\alpha^-}(\tau^+ \ast \rho)(x)}e^{\alpha^- - \alpha^+} &< 1
   \\ \sup \limits_{x \in \R^d}\ e^{e^{\alpha^-}(\kappa^- \ast \rho)(x)}e^{e^{\alpha^+}(\tau^- \ast \rho)(x)}e^{\alpha^+ - \alpha^-} &< 1.
  \end{align*}
\end{enumerate}

\subsection{Scaling}
Put $z^{\pm} \longmapsto n z^{\pm}$ and scale the potentials by $\frac{1}{n}$, i.e. $g \longmapsto \frac{1}{n}g$ where 
$g \in \{\phi^{\pm}, \psi^{\pm}, \kappa^{\pm}, \tau^{\pm}\}$. Denote by $L_{n}$ the corresponding scaled Markov operator. 
Let 
\[
 M_n(\eta) = |\eta^+| + |\eta^-| + \sum \limits_{x \in \eta^+}e^{-n E_{\kappa^+}(x, \eta^+ \backslash x)}e^{-n E_{\tau^+}(x,\eta^-)} 
 + \sum \limits_{x \in \eta^-}e^{-nE_{\kappa^-}(x,\eta^- \backslash x)}e^{-n E_{\tau^-}(x,\eta^+)}
\]
be the scaled cumulative death intensity and
\[
 D(\widehat{L}_n) := \{ G \in \Lb_{\alpha} \ | \ M_n \cdot G \in \Lb_{\alpha} \}.
\]
Let $\widehat{L}_n := K^{-1}L_nK$ be defined on $D(\widehat{L}_n)$ and put $\widehat{L}_{n,\mathrm{ren}} := R_{n}\widehat{L}_nR_{n^{-1}}$.
Then, $\widehat{L}_{n, \mathrm{ren}} = A_n + B_n$ where $(A_nG)(\eta) = -M_n(\eta)G(\eta)$.
Let $f^n_x(g;\eta) := \prod \limits_{y \in \eta}n \left( e^{-n g(x - y)} - 1\right)$, the second operator is given by
\begin{align*}
 &\ (B_nG)(\eta) = z^+ \sum \limits_{\xi \subset \eta}\int \limits_{\R^d}e^{-nE_{\phi^+}(x,\xi^+)}e^{-nE_{\psi^+}(x,\xi^-)}f^n_x(\phi^+;\eta^+ \backslash \xi^+)f_x^n(\psi^+;\eta^- \backslash \xi^-)G(\xi^+ \cup x, \xi^-)\dm x
  \\ &\ \ \ + z^- \sum \limits_{\xi \subset \eta}\int \limits_{\R^d}e^{-nE_{\phi^-}(x,\xi^-)}e^{-nE_{\psi^-}(x,\xi^+)}f^n_x(\phi^-;\eta^- \backslash \xi^-)f_x^n(\psi^-;\eta^+ \backslash \xi^+)G(\xi^+, \xi^- \cup x)\dm x
  \\ &\ \ \ + \sum \limits_{\xi \subset \eta}\sum \limits_{x \in \xi^+}e^{-nE_{\kappa^+}(x,\xi^+ \backslash x)}e^{-nE_{\tau^+}(x,\xi^-)}f^n_x(\kappa^+;\eta^+ \backslash \xi^+)f_x^n(\tau^+;\eta^- \backslash \xi^-)G(\xi^+ \backslash x, \xi^- \cup x)
  \\ &\ \ \ + \sum \limits_{\xi \subset \eta}\sum \limits_{x \in \xi^-}e^{-nE_{\kappa^-}(x,\xi^- \backslash x)}e^{-nE_{\tau^-}(x,\xi^+)}f_x^n(\kappa^-;\eta^- \backslash \xi^-)f_x^n(\tau^-;\eta^+ \backslash \xi^+)G(\xi^+ \cup x, \xi^- \backslash x)
  \\ &\ \ \ - \sum \limits_{\bfrac{\xi \subset \eta}{\xi \neq \eta}}\sum \limits_{x \in \xi^+}e^{-nE_{\kappa^+}(x,\xi^+ \backslash x)}e^{-nE_{\tau^+}(x,\xi^-)}f^n_x(\kappa^+;\eta^+ \backslash \xi^+)f_x^n(\tau^+;\eta^- \backslash \xi^-)G(\xi)
  \\ &\ \ \ - \sum \limits_{\bfrac{\xi \subset \eta}{\xi \neq \eta}}\sum \limits_{x \in \xi^-}e^{-nE_{\kappa^-}(x,\xi^- \backslash x)}e^{-nE_{\tau^-}(x,\xi^+)}f^n_x(\kappa^-;\eta^- \backslash \xi^-)f_x^n(\tau^-;\eta^+ \backslash \xi^+)G(\xi).
\end{align*}
Then, $(\widehat{L}_{n,\mathrm{ren}}, D(\widehat{L}_n))$ is a well-defined operator on $\Lb_{\alpha}$.
The next statement follows by the same arguments as Theorem \ref{POTTSTH:00} and Theorem \ref{POTTSTH:01}.
\begin{Theorem}\label{POTTSTH:06}
 Let $n \in \N$ be arbitrary and fixed. The following assertions are satisfied.
 \begin{enumerate}
  \item $(\widehat{L}_{n,\mathrm{ren}}, D(\widehat{L}_{n,\mathrm{ren}}))$ is the generator of an analytic semigroup 
   $(\widehat{T}_{n,\mathrm{ren}}(t))_{t \geq 0}$ of contractions on $\Lb_{\alpha}$. Moreover, $B_{bs}(\Gamma_0^2)$ is a core.
  \item Let $(\widehat{T}_{n,\mathrm{ren}}(t)^*)_{t \geq 0}$ be the adjoint semigroup. For any $k_0 \in \K_{\alpha}$ there exists a unique weak solution to
  \[
   \frac{\dm}{\dm t}\langle G, k_{t,n}\rangle = \langle \widehat{L}_{n,\mathrm{ren}}G, k_{t,n}\rangle, \ \ k_{t,n}|_{t=0} = k_0, \ \ G \in B_{bs}(\Gamma_0^2).
  \]
  This solution is given by $k_{t,n} = \widehat{T}_{n,\mathrm{ren}}(t)^*k_0$.
 \end{enumerate}
\end{Theorem}
In the next step we consider the limiting operators, as $n \to \infty$. These operators are formally given by
$\widehat{L}_V = A_V + B_V$, where $A_VG(\eta) = - M_V(\eta)G(\eta)$ and
\begin{align*}
 &\ (B_VG)(\eta) = z^+ \sum \limits_{\xi \subset \eta}\int \limits_{\R^d}e_{\lambda}(-\phi^+(x-\cdot); \eta^+ \backslash \xi^+)e_{\lambda}(-\psi^+(x-\cdot);\eta^- \backslash \xi^-)G(\xi^+ \cup x, \xi^-)\dm x
  \\ &\ \ \ + z^- \sum \limits_{\xi \subset \eta}\int \limits_{\R^d}e_{\lambda}(-\phi^-(x-\cdot);\eta^- \backslash \xi^-)e_{\lambda}(-\psi^-(x-\cdot);\eta^+ \backslash \xi^+)G(\xi^+, \xi^- \cup x)\dm x
  \\ &\ \ \ + \sum \limits_{\xi \subset \eta}\sum \limits_{x \in \xi^+}e_{\lambda}(-\kappa^+(x-\cdot);\eta^+ \backslash \xi^+)e_{\lambda}(-\tau^+(x-\cdot);\eta^- \backslash \xi^-)G(\xi^+ \backslash x, \xi^- \cup x)
  \\ &\ \ \ + \sum \limits_{\xi \subset \eta}\sum \limits_{x \in \xi^-}e_{\lambda}(-\kappa^-(x-\cdot);\eta^- \backslash \xi^-)e_{\lambda}(-\tau^-(x-\cdot);\eta^+ \backslash \xi^+)G(\xi^+ \cup x, \xi^- \backslash x)
  \\ &\ \ \ - \sum \limits_{\bfrac{\xi \subset \eta}{\xi \neq \eta}}\sum \limits_{x \in \xi^+}e_{\lambda}(-\kappa^+(x-\cdot);\eta^+ \backslash \xi^+)e_{\lambda}(-\tau^+(x-\cdot);\eta^- \backslash \xi^-)G(\xi)
  \\ &\ \ \ - \sum \limits_{\bfrac{\xi \subset \eta}{\xi \neq \eta}}\sum \limits_{x \in \xi^-}e_{\lambda}(-\kappa^-(x-\cdot);\eta^- \backslash \xi^-)e_{\lambda}(-\tau^-(x-\cdot);\eta^+ \backslash \xi^+)G(\xi).
\end{align*}
Here, $M_V(\eta) = 2|\eta^+| + 2|\eta^-|$ and $\widehat{L}_V$ is a well-defined operator on $\Lb_{\alpha}$ with domain
\[
 D(\widehat{L}_V) := \{ G \in \Lb_{\alpha} \ | \ M_V \cdot G \in \Lb_{\alpha} \}.
\]
\begin{Theorem}\label{POTTSTH:07}
 The following assertions are satisfied:
 \begin{enumerate}
  \item The operator $(\widehat{L}_V, D(\widehat{L}_V))$ is the generator of an analytic semigroup $(\widehat{T}^V(t))_{t \geq 0}$ of contractions on $\Lb_{\alpha}$.
  Moreover, $B_{bs}(\Gamma_0^2)$ is a core for the generator.
  \item Let $(\widehat{T}^V(t)^*)_{t \geq 0}$ be the adjoint semigroup on $\K_{\alpha}$. Then, for any $r_0 \in \K_{\alpha}$ there exists a unique solution to
  \begin{align}\label{POTTS:07}
   \frac{\dm}{\dm t}\langle G, r_t\rangle = \langle \widehat{L}_VG, r_t \rangle, \ \ r_t|_{t=0} = r_0, \ \ G \in B_{bs}(\Gamma_0^2).
  \end{align}
  The solution is given by $r_t = \widehat{T}^V(t)^*r_0$.
  \item Let $r_0(\eta) = \prod \limits_{x \in \eta^+}\rho_0^+(x)\prod \limits_{x \in \eta^-}\rho_0^-(x)$ with
  \[
   \rho_0^{\pm}(x) \leq e^{\alpha^{\pm}}\rho(x), \ \ x \in \R^d.
  \]
  Assume that $(\rho_t^+, \rho_t^-)$ is a classical solution to 
  \begin{align}
  \label{VLASOV1} \frac{\partial \rho_t^+(x)}{\partial t} &= -\left(1 + e^{-(\kappa^+ \ast \rho_t^+)(x)}e^{-(\tau^+\ast \rho_t^-)(x)}\right) \rho_t^+(x) 
  \\ \notag &\ \ \ + z^+ e^{-(\phi^+ \ast \rho_t^+)(x)}e^{-(\psi^+\ast \rho_t^-)(x)} + e^{-(\kappa^- \ast \rho_t^-)(x)}e^{-(\tau^- \ast \rho_t^+)(x)} \rho_t^-(x)
  \\ \label{VLASOV2} \frac{\partial \rho_t^-(x)}{\partial t} &= - \left(1 + e^{-(\kappa^- \ast \rho_t^-)(x)}e^{-(\tau^- \ast \rho_t^+)(x)}\right) \rho_t^-(x) 
  \\ \notag &\ \ \ + z^- e^{-(\phi^- \ast \rho_t^-)(x)}e^{-(\psi^- \ast \rho_t^+)(x)} + e^{-(\kappa^+ \ast \rho_t^+)(x)}e^{-(\tau^+\ast \rho_t^-)(x)}\rho_t^+(x).
 \end{align}
  such that
  \[
   \rho_t^{\pm}(x) \leq A e^{\alpha^{\pm}}\rho(x), \ \ x \in \R^d, \ \ t \geq 0
  \]
  holds for some constant $A > 0$.
  Then, $r_t(\eta) := \prod \limits_{x \in \eta^+}\rho_t^+(x) \prod \limits_{x \in \eta^-}\rho_t^-(x)$ is a weak solution to \eqref{POTTS:07}.
 \end{enumerate}
\end{Theorem}
\begin{proof}
 The first two assertions follow by a modification of the arguments given in the proof of Theorem \ref{POTTSTH:00} and Theorem \ref{POTTSTH:01}.
 For the last assertion, observe that $r_t(\eta)$ is continuous w.r.t. $\sigma(\K_{\alpha}, \Lb_{\alpha})$. 
 Since $\Vert r_t \Vert_{\K_{\alpha}} \leq A$ for all $t \geq 0$, it follows that $r_t$ is continuous w.r.t. $\mathcal{C}$.
 The adjoint operator to $\widehat{L}_V$ is given by
 \begin{align*}
  (L_V^{\Delta}k)(\eta) = &- |\eta|k(\eta) - \sum \limits_{x \in \eta^-}\mathcal{Q}^V_x(\kappa^+, \tau^+)k(\eta)
  - \sum \limits_{x \in \eta^+}\mathcal{Q}^V_x(\tau^-,\kappa^+)k(\eta)
  \\ &+ z^+ \sum \limits_{x \in \eta^+}\mathcal{Q}^V_x(\phi^+,\psi^+)k(\eta^+ \backslash x, \eta^-)
      + z^- \sum \limits_{x \in \eta^-}\mathcal{Q}^V_x(\psi^-, \phi^-)k(\eta^+, \eta^- \backslash x)
  \\ &+ \sum \limits_{x \in \eta^+}\mathcal{Q}^V_x(\kappa^+, \tau^+)k(\eta^+ \backslash x, \eta^-)
      + \sum \limits_{x \in \eta^-}\mathcal{Q}^V_x(\tau^-,\kappa^-)k(\eta^+, \eta^- \backslash x)
 \end{align*}
 defined on its maximal domain $D(L_V^{\Delta}) = \{ k \in \K_{\alpha} \ | \ L_V^{\Delta}k \in \K_{\alpha}\}$ and
 \[
  \mathcal{Q}^V_x(g_0,g_1)k(\eta) = \int \limits_{\Gamma_0^2}e_{\lambda}(-g_0(x-\cdot);\xi^+)e_{\lambda}(-g_1(x-\cdot);\xi^-)k(\eta \cup \xi)\dm \lambda(\xi).
 \]
 We have
 \[
  \frac{\partial r_t(\eta)}{\partial t} = \sum \limits_{x \in \eta^+}r_t(\eta^+ \backslash x, \eta^-)\frac{\partial \rho_t^+(x)}{\partial t}
  + \sum \limits_{x \in \eta^-}r_t(\eta^+, \eta^- \backslash x)\frac{\partial \rho_t^-(x)}{\partial t}.
 \]
 An easy computation (see e.g. \cite{FFHKKK15}) shows that $r_t$ solves \eqref{POTTS:07}, provided $(\rho_t^+, \rho_t^-)$ solve \eqref{VLASOV1} and \eqref{VLASOV2}.
\end{proof}
The next statement establishes convergence of the scaled evolution to the limiting solutions.
\begin{Theorem}\label{POTTSTH:08}
 Let $G \in \Lb_{\alpha}$, then, $\widehat{T}_{n,\mathrm{ren}}(t)G \longrightarrow \widehat{T}^V(t)G$ as $n \to \infty$.
 In particular, for any $r_{0}$ we have
 \[
  \langle G, \widehat{T}_{n,\mathrm{ren}}(t)^*r_0 \rangle \longrightarrow \langle G, \widehat{T}_V(t)^*r_0 \rangle, \ \ n \to 0, \ \ G \in \Lb_{\alpha}.
 \]
\end{Theorem}
\begin{proof}
 Since $B_{bs}(\Gamma_0^2)$ is a core for $\widehat{L}_{n,\mathrm{ren}}$ and $\widehat{L}_V$, it suffices to show that
 \[
  \widehat{L}_{n,\mathrm{ren}}G \longrightarrow \widehat{L}_VG, \ \ n \to \infty
 \]
 holds in $\Lb_{\alpha}$, for any $G \in B_{bs}(\Gamma_0^2)$. But this follows, by dominated convergence, similarly to \cite{FFHKKK15}.
\end{proof}

\subsection*{Acknowledgments}
Financial support through CRC701, project A5, at Bielefeld University is gratefully acknowledged.

\newpage

\begin{footnotesize}

\bibliographystyle{alpha}
\bibliography{Bibliography}

\end{footnotesize}

\end{document}